\documentclass[11pt]{article}
\usepackage[utf8]{inputenc}
\usepackage{amsmath}
\usepackage{amsfonts}
\usepackage{amssymb}
\usepackage{amsthm}
\usepackage{amsrefs}
\usepackage{mathrsfs}
\usepackage[all]{xy}
\usepackage{tikz-cd}
\usepackage{xcolor}

\theoremstyle{plain}
  \newtheorem{thm}{Theorem}[section]
  \newtheorem{lem}[thm]{Lemma}
  \newtheorem{prop}[thm]{Proposition}
   
\theoremstyle{definition}
  
  \newtheorem{rmk}[thm]{Remark}
  \newtheorem{ex}[thm]{Example}

\theoremstyle{plain}
  
\DeclareMathOperator{\tC}{Cone} 
\DeclareMathOperator{\Lie}{Lie}
\numberwithin{equation}{section}
\allowdisplaybreaks
\def\CA{\mathcal{A}}

\def\CF{\mathcal{F}}
\def\CC{\mathcal{C}}

\def\th{\theta}

\def\om{\omega}
\def\Om{\Omega}

\def\w{\wedge}

\def\tr{\mathrm{tr}}
\def\DC{D_{\mathcal{C}}}
\def\FC{\mathcal{F}_{\mathcal{C}}}

\def\pa{\partial}

\newcommand{\overbar}[1]{\mkern 2.5mu\overline{\mkern-2.5mu#1\mkern 1.5mu}\mkern-1.5mu}

\begin{document}

\title{
\bf\
{Mapping Cone Connections and their Yang-Mills Functional}
}

\author{Li-Sheng Tseng and Jiawei Zhou\\
\\
}

\date{}

\maketitle

\begin{abstract} 
For a given closed two-form, we introduce the cone Yang-Mills functional  which is a Yang-Mills-type functional for a pair $(A,B)$, a connection one-form $A$ and a scalar $B$ taking value in the adjoint representation of a Lie group.  The functional arises naturally from dimensionally reducing the Yang-Mills functional over the fiber of a circle bundle with the two-form being the Euler class.  We  write down the Euler-Lagrange equations of the functional and present some of the properties of its critical solutions, especially in comparison with Yang-Mills solutions.  We show that a special class of three-dimensional solutions satisfy a duality condition  which  generalizes the Bogomolny monopole equations. Moreover, we analyze the zero solutions of the cone Yang-Mills functional and give an algebraic classification characterizing principal bundles that carry such cone-flat solutions when the two-form is non-degenerate.

\end{abstract}

\tableofcontents

\section{Introduction}

Let $(M^m,g)$ be a Riemannian manifold and $P$ a principal $G$-bundle over $M$.  For simplicity, we will assume that the Lie group $G$ is compact and a subgroup of $SO(N)$. 

In this paper, we study a Yang-Mills type functional that involves a pair $(A, B)$, with $A$ being a connection form on the associated adjoint bundle $Ad\,P$,
a section $B\in \Om^0(M, Ad\,P)$, 
and also a two-form $\zeta\in \Om^2(M)$ that is $d$-closed.
We shall call this functional the \textit{cone Yang-Mills functional} and it is defined by 
\begin{align}\label{cym}
S_{cY\!M}(A,B)&= \|F_A+\zeta B\|^2 + \|d_A B\|^2 \nonumber\\ 
&= c_G \int_M  \tr \left[(F_A + \zeta B) \w * (F_A + \zeta B) \,+ \, d_A B \w * d_A B \right],
\end{align}
where $F_A= dA + A\w A$ is the curvature, $d_A B = dB+ [A,B]$ is the covariant derivative of $B$, and in the second line, we have inserted the Killing form of the Lie group $G$ which for $SO(N)$ and $SU(N)$ is given by the trace times a constant dependent on $N$, denoted by $c_G$. 
The critical points of the functional are solutions of the associated Euler-Lagrange equations
\begin{align}
d_A^*(F_A+\zeta B)+[B, d_AB] &= 0\,, \label{Afeq}\\
\zeta^*\!\left(F_A+\zeta B\right)+d_A^*d_AB &=0\,,\label{Bfeq}
\end{align}
where $\zeta^*\!:\Om^k(M)\to\Om^{k-2}(M)$ is the adjoint of the wedging map, $\zeta\w \,$, and can be explicitly expressed as 
\begin{align}
\zeta^*=(-1)^{(m-k)k}*\zeta*\,.
\end{align}
Let us give both a physical and a mathematical motivation for considering the cone Yang-Mills functional.  

For the physical motivation, the cone Yang-Mills functional can be considered as the dimensional reduction of the standard Yang-Mills functional over the $S^1$ fiber of a circle bundle, $\pi: X \to M$.  Unlike a product space, a circle bundle generically is one where the fiber circle is non-trivially twisted over $M$.  Such arises for instance in Kaluza-Klein monopole solutions with a non-vanishing background two-form flux $\zeta$.  
On any circle bundle $X$, there exists what is called a global angular one-form $\theta$ which can be locally expressed as $\theta= dz + a$ where $z$ denotes the $S^1$ fiber coordinate and $a$ is the $U(1)$ Kaluza-Klein gauge field.  The two-form flux, $\zeta = d\theta = d(dz+a)=da$, is then the background field strength of the $U(1)$ gauge field and effectively measures the twisting of the $S^1$ fiber.  

Now for the dimensional reduction of the Yang-Mills functional over the $S^1$ fiber of the circle bundle $\pi: X \to M$, 
we take the metric to be
\begin{align}\label{Xmetric}
g_X = \pi^* g_M + \theta \otimes \theta\,,
\end{align}
and further require that the local connection one-form (Yang-Mills gauge field) $\CA$ on $X$ be invariant under translation of the circle fiber.  Such a connection form on $X$ can be expressed as
\begin{align}\label{Xconn}
\CA = A + \theta\, B  \,.   
\end{align}
Here, both $A$ and $B$  takes values on $Ad\, P$, and $B$ represents the component of the connection in the circle direction. Since $\CA$ is invariant under $S^1$ translation, both $A$ and $B$ only have dependence on the coordinates of $M$.  The curvature two-form of $\CA$ (Yang-Mills field strength) on $X$ then has the following expression:
\begin{align}
\CF_\CA = d \CA + \CA \w \CA  = (F_A + \zeta B) + \theta (-d_A B)\,,
\end{align}
where we have used $d\theta=\zeta$.  Comparing with the cone Yang-Mills functional in \eqref{cym}, we see that $S_{cYM}=\|\CF_\CA\|^2\,$,  and therefore, it is just the Yang-Mills functional on $X$ dimensionally reduced to $M$ over the fiber circle.  And we can also interpret the cone Yang-Mills Euler-Lagrange equations \eqref{Afeq}-\eqref{Bfeq} as the dimensionally-reduced Yang-Mills equations over the $S^1$ fiber of $X$.

In physical terms, the cone Yang-Mills functional on $M$ is a Euclidean or Wick-rotated action, situated in a curved background, with Riemannian metric $g$: 
\begin{align*}
S_{cYM}(A,B)&=\int_M c_G \, \tr\left(|F_A + \zeta B|^2 + |d_A B|^2 \right) dV\\
&=\int_M c_G\, \tr \left(|F_A|^2 + |d_A B|^2 + |\zeta|^2B^2 + 2 \,\zeta^* F_A B \right) dV\,, 
\end{align*}
where we have used a simplified notation, e.g. $F_A  \w * F_A = |F_A|^2 dV$.  
The action thus consists of a Yang-Mills part, an adjoint-valued scalar with a ``mass" $|\zeta|$ (which may vary over $M$), and an interaction term between $A$ and $B$.  Heuristically, the Euler-Lagrange equations of motion \eqref{Afeq}-\eqref{Bfeq} can be thought of as a mixing of the Yang-Mills equations with a Klein-Gordon-type scalar plus an interaction term.

The cone Yang-Mills functional 
also has a mathematical motivation. 
In fact, the functional arises naturally from its relation with a mapping cone complex which we will provide an explanation here.  (For a reference on the general mapping cone complex, see for example \cite{Weibel} and also \cite{CTT}.)  

The $d$-closed two-form $\zeta\in \Om^2(M)$ can be thought of as an operator or a \textit{map}, by wedge product between differential forms, i.e. $\zeta\w : \Omega^*(M) \to \Omega^{*+2}(M)$.  
If desired, the form $\zeta$ can represent a geometric structure of interest on $M$, such as a symplectic or hermitian structure, but generally, the form $\zeta$ can be any closed two-form on $M$.  
Letting $\zeta$ map between two de Rham chain complexes leads to a mapping cone complex (see Figure \ref{mccomp}).
\begin{figure}[t]
\begin{equation*}
\begin{tikzcd}
\ldots \arrow[r, "d_C"] & \textcolor{purple}{\tC^k(\zeta)} \arrow[r, "d_C"] & \textcolor{blue}{\tC^{k+1}(\zeta)} \arrow[r, "d_C"] & \textcolor{red}{\tC^{k+2}(\zeta)} \arrow[r, "d_C"] & \ldots & \\
~\ldots \arrow[r, "d"] & \textcolor{purple}{\Om^{k}(M)} \arrow[r, "d"] & \textcolor{blue}{\Om^{k+1}(M)} \arrow[r, "d"]  & \textcolor{red}{\Om^{k+2}(M)}\arrow[r, "d"]  & \ldots~~  \\
 ~~\ldots  \arrow[r, "-d"]  & \textcolor{purple}{\Om^{k-1}(M)}  \arrow[r, "-d"] \arrow[ur, "~\zeta\, \w"] &  \textcolor{blue}{\Om^{k}(M)} \arrow[r, "-d"] \arrow[ur, "~\zeta\, \w"] & \textcolor{red}{\Om^{k+1}(M)}  \arrow[r, "-d"] & \ldots
\end{tikzcd}
\end{equation*}
\caption{The mapping cone complex of the $d$-closed two-form $\zeta$.  The complex of elements $\tC^*(\zeta)=\Om^*(M)\oplus \Om^{*-1}(M)$ arises from the $\zeta$ map between two de Rham complexes.}\label{mccomp}
\end{figure}
The elements of the mapping cone complex consist of pairs of  differential forms:  
\begin{align}\label{Conedef}
\tC^k(\zeta) : &= \Om^k(M) \oplus \Om^{k-1}(M)\,, \qquad  k=0, 1, \ldots, m +1\,.
\end{align}
The differential of the mapping cone complex $d_C: \tC^k(\zeta)\to \tC^{k+1}(\zeta)$ is given by
\begin{align*}
d_C \tC^k(\zeta)&= d_C \left(  \Omega^k \oplus \Omega^{k-1}\right) \\
&= \left(d\,\Om^k + \zeta \w \Omega^{k-1}\right) \oplus -d\,\Omega^{k-1}\,.
\end{align*}
Since the grading of the two components of $\tC^k(\zeta)$ in \eqref{Conedef} are different, it is useful to introduce a formal one-form $\theta$ to supplement the second component so that we can express a cone form as a sum with both components having the same total degree $k$, i.e.
\begin{align}\label{thdef}
\tC^k(\zeta)&= \Om^k \oplus\, \theta\w \Om^{k-1}\,.
\end{align}
Additionally, it is useful for us to impose that the formal one-form $\theta$ satisfy $d\theta = \zeta$.  For this will allow us to interpret the cone differential $d_C$ simply as the exterior derivative:
\begin{align*} 
d_C \tC^k(\zeta) &= d \left(\Om^k \oplus\, \theta \w \Om^{k-1}\right)\\
&=\left(d\,\Om^k + \zeta \w \Omega^{k-1}\right) \oplus\,  \theta \w \left(-d\,\Omega^{k-1}\right)\,.
\end{align*}
It then becomes self-evident that $d_C\,d_C =0\,.$ 
Moreover, there is a natural product on the cone space $\tC^*(\zeta)$ given by the usual wedge product
\begin{align*}
\tC^j(\zeta) \times \tC^k(\zeta) :&= \left(\Om^j \oplus\, \theta \w \Om^{j-1}\right) \w \left(\Om^k \oplus\, \theta \w \Om^{k-1} \right)\\
&=\left(\Om^j \w \Om^k\right) \oplus \,\theta \w \left(\Om^{j-1} \w \Om^k + (-1)^j \Om^j \w \Om^{k-1}\right)\,.
\end{align*}
This product satisfies the Leibniz rule with respect to $d_C$.  In all, we see that we have a mapping cone algebra, $(\tC^*(\zeta), d_C, \times)$, that satisfies the  conditions of a differential graded algebra (DGA).

Now in the presence of the associated adjoint bundle, we should consider the twisted cone forms: 
\begin{align*}
\tC^k(\zeta)(M, Ad\,P)=\Om^k(M, Ad\,P) \oplus\, \theta \w  \Om^{k-1}(M, Ad\,P)
\end{align*}
which take values on the associated adjoint bundle $Ad\,P$.  In this context, the differential $d_C$ must also be twisted by a cone connection form $\CA$
\begin{equation*}
\begin{tikzcd}
\ldots \arrow[r, "d_C+\CA"] & {\tC^k(\zeta)(M, Ad\,P)} \arrow[r, "d_C+\CA"] & {\tC^{k+1}(\zeta)(M, Ad\,P)} \arrow[r, "d_C+\CA"] & {\tC^{k+2}(\zeta)(M, Ad\,P)} \arrow[r, "d_C+\CA"] & \ldots 
\end{tikzcd}
\end{equation*}
where $\CA = A + \theta\w B\,$, 
with $A$ being the connection form on $M$ 
and $B \in \Om^0(M, Ad\, P)$. The cone curvature then takes the form
\begin{align}\label{Ccurv}
\CF_\CA = (d_C+ \CA)^2= d_C \CA + \CA \w \CA  = (F_A + \zeta B) + \theta\w (-d_A B)\,,
\end{align}
where $F_A = d A + A\w A$.  The above twisted complex is only a differential complex if $\CF_\CA =0$.  This requires that the cone curvature $\CF_\CA$ vanishes, which from \eqref{Ccurv} corresponds to $(A, B)$ satisfying what we shall call the cone-flat condition with respect to $\zeta$, or simply just the cone-flat condition,
\begin{align}\label{coneflat}
F_A + \zeta\, B =0\,,  \qquad  d_A\, B = 0\,.   
\end{align}  
The cone Yang-Mills functional of \eqref{cym} is then just the normed square of the cone curvature, $\|\CF_\CA \|^2$.  It is worthwhile to emphasize that this mapping cone perspective only requires that the two-form $\zeta\in \Om^2(M)$ be $d$-closed and nothing more.  
This is in contrast with the dimensional reduction perspective, where $\zeta$ mathematically represents the Euler class of the circle bundle $\pi: X\to M$ and hence would need to be an element of $H^2(M, \mathbb{Z})$.

In this paper, we take an important first step in understanding cone Yang-Mills solutions.  We will show that a subset of the solutions involves Yang-Mills connections.  For instance, in dimension two when $M$ is compact and $\zeta$ is taken to be the volume form, the above cone-flat condition \eqref{coneflat} implies exactly the two-dimensional Yang-Mills condition
\begin{align}\label{2DYM}
d_A^*\, F_A = -*\,d_A *\, F_A=*\, d_A B = 0\,, 
\end{align}
having noted that $*\, \zeta  = 1$.  Conversely, in two dimensions, if $(A, B)$ is a cone-Yang-Mills solution with $A$ also a Yang-Mills connection, then the pair $(A,B)$ must be cone-flat.

In higher dimensions, certain special classes of Yang-Mills connections can be paired with a scalar $B$ to obtain cone Yang-Mills solutions. In the trivial case where we set $\zeta=0$, any Yang-Mills connection $A$ together with a covariantly constant $B$ is trivially a cone Yang-Mills solutions.  When $\zeta \neq 0$, Yang-Mills connections $A$ such that the curvature two-form satisfy $\zeta^*F_A=0$ are cone Yang-Mills solutions with $B=0$.  When $\zeta$ is a harmonic form, cone-flat solutions are always composed of a Yang-Mills connection with an appropriate scalar section $B$.
  
On the other hand, it should be evident that the space of cone Yang-Mills solutions $(A, B)$ is generally much richer and different from that of Yang-Mills solutions.  As we will show in explicit examples, not all Yang-Mills connection $A$ can be paired with a scalar $B$ to form a cone Yang-Mills solutions.  Conversely, there are also cone Yang-Mills solutions $(A,B)$ where $A$ is not Yang-Mills.  Furthermore, given a cone Yang-Mills solutions $(A, B)$, there may be other scalars $B'$ such that $(A,B')$ remain cone Yang-Mills. 

Interestingly, when $M$ is three-dimensional, we are able to write down a Bogomolny-type condition, that gives a sufficient condition on $(A, B)$ to be a cone Yang-Mills solution.  Such a condition can be motivated by considering $\CF_\CA$ as the Yang-Mills curvature on a four-dimensional circle bundle $X$ and imposing the (anti-)self-dual condition.  Then, dimensionally reducing on the fiber $S^1$ by expressing $\CF_\CA$ as in  \eqref{Ccurv}, the (anti-)self-dual condition implies the three-dimensional condition 
\begin{align}
\label{Beqng}
F_A + \zeta B = \pm * d_A B\,.
\end{align}
Notice that if we set $\zeta=0$, the above equation is just the Bogomolny monopole equation. 
Similar to the (anti-)self-dual Yang-Mills condition in four dimensions, \eqref{Beqng} is a first-order condition whose solutions solve \eqref{Afeq}-\eqref{Bfeq} in three dimensions.  This duality condition will assist us in finding non-abelian, cone Yang-Mills solutions.  Indeed, we shall give in Section \ref{3dsol} an explicit $SU(2)$ solution of \eqref{Beqng} that comes from dimensionally reducing the four-dimensional Taub-NUT gravitational instanton solution.

Concerning the cone-flat condition, we are able to characterize principal bundles that carry a cone-flat connection, i.e. a pair  $(A,B)$ satisfying \eqref{coneflat}, when $\zeta$ is a non-degenerate two-form. This type of cone-flat connections can interestingly be classified similar to Atiyah-Bott's classification of bundles carrying Yang-Mills connections on Riemann surfaces \cite{Atiyah-Bott}.  Our classification of the cone-flat bundles however depends on the given $\zeta$ and the second homotopy group, $\pi_2(M)$.
\begin{thm}\label{coneflatThrm}
Let $M$ be a path connected manifold, $\zeta\in\Omega^2(M)$ be a non-zero, non-degenerate, closed two-form, and $G$ be a Lie group. There exists a bijective correspondence between the following sets:
\begin{align*}
\left\{\begin{matrix}\text{isomorphism classes of cone-flat connections}
\\ \text{with respect to $\zeta$ on  $G$-bundles over $M$}
\end{matrix}\right\}
\simeq
\left\{\begin{matrix}\text{conjugacy classes of}\\ \text{homomorphisms }
\rho: \Gamma\to G\end{matrix}\right\},
\end{align*}
where $\Gamma$ is an $\mathbb{R}/\overbar{H}$ extension of $\pi_1(M)$ with $\overbar{H}\subset\mathbb{R}$ being the closure of the group 
$$
H:=\left\{ \int_\mathcal{S}\zeta \ \big| \  \mathcal{S} \text{ is a representative in } \pi_2(M) \right\}.
$$
\end{thm}

This paper is organized as follows. In Section 2, after a brief description of our notations/conventions, we proceed to consider the first-order variation of the cone Yang-Mills functional to obtain its Euler-Lagrange equations.  We also show that modulo gauge equivalence, the cone Yang-Mills equations are elliptic and hence has a finite-dimensional solution space on a closed manifold.  We also describe properties of cone Yang-Mills solutions under certain conditions for the two-form $\zeta$ and structure group $G$, and especially emphasizing its relationship to Yang-Mills connections.  In Section 3, we work out the special case of abelian cone Yang-Mills solutions in dimension two.  We also discuss the three-dimensional Bogomolny-type monopole condition \eqref{Beqng} and give an explicit non-trivial $SU(2)$ cone Yang-Mills solutions on $M=\mathbb{R}^3 - \{0\}$, i.e. the Euclidean space with the origin removed.   Finally, in Section 4, we consider bundles that can carry cone-flat solutions when $\zeta$ is a non-degenerate, closed two-form, and prove the classification of cone-flat bundles of Theorem \ref{coneflatThrm}.

\

\noindent{\it Acknowledgements.~} 
We would like to thank Alberto Cattaneo, David Clausen, Jonathan Delgado, Bernard Julia, Claude LeBrun, Si Li, Jianfeng Lin, Xiang Tang, Daniel Waldram, Jiaping Wang, Richard Wentworth, and Dan Xie for helpful comments and suggestions.  We would like to acknowledge the support of the Simons Collaboration Grant No.~636284 for the first author and the support of the National Key Research and Development Program of China No.~2020YFA0713000 for the second author.

\section{Properties of cone Yang-Mills solutions}

\subsection{Preliminaries}

Let $M$ be a smooth manifold and $P$ be a principal bundle over $M$.  We consider the associated adjoint bundle $Ad\,P$ equipped with an inner product that is invariant under the adjoint action.  
For the inner product on $\Om^*(M,Ad\,P)$, we use the standard one induced from the Riemannian metric $g$ and the Killing form of the Lie group $G$, i.e. for $\eta_1,\eta_2 \in\Omega^*(M,Ad\,P)$,
\begin{align}\label{met0}
\langle\eta_1,\eta_2\rangle = c_G \int_M \tr \left[ \eta_1 \w *\eta_2 \right]\,,
\end{align}
where the constant $c_G=-(N-2)$ if $G=SO(N)$.  As will be useful later, we note here for $\eta\in \Omega^0(M,Ad\,P)$, we have the relation
\begin{align}\label{adjinp}
\langle [\eta_1,\eta],\eta_2\rangle=\langle \eta_1,[\eta,\eta_2]\rangle.
\end{align}

We can extend the inner product $\langle-,-\rangle$ on $\Om^*(M,Ad\,P)$ to $\langle-,-\rangle_{\mathcal{C}}$ on $\tC^*(\zeta)(M,Ad\,P)=\Om^*(M,Ad\,P) \oplus \theta \Om^{*-1}(M,Ad\,P)$ for a $d$-closed two-form, $\zeta\in \Om^2(M)$.   We do so by setting
\begin{align} \label{cmet1}
\langle \eta_1+\theta\xi_1, \eta_2+\theta\xi_2\rangle_{\mathcal{C}}=\langle\eta_1,\eta_2\rangle+\langle\xi_1,\xi_2\rangle.
\end{align}
Note that $\theta$ is a formal one-form with property $d\theta=\zeta$.  When paired with another differential form, there is a wedge product, $\theta \w$, which for notational simplicity, as in \eqref{cmet1}, we will just assume without writing out explicitly.

With \eqref{met0} and \eqref{cmet1}, we can also express $\langle-,-\rangle_{\mathcal{C}}$ as an integral by introducing the Hodge star operator for cone forms
\begin{align} \label{conestar}
*_{\mathcal{C}}:\tC^k(\zeta)(M,Ad\,P)~&\to~~\tC^{m+1-k}(\zeta)(M,Ad\,P)\nonumber\\ 
\eta+\theta\xi~\qquad&\mapsto~~ *\xi + \theta\, (-1)^{|\eta|} \!*\eta 
\end{align}
which allows us to write
\begin{align}\label{cmet2}
\langle \eta_1+\theta\xi_1, \eta_2+\theta\xi_2\rangle_{\mathcal{C}}=c_G \int_M \tr \, \frac{\partial}{\partial\theta}[(\eta_1+\theta\xi_1)\wedge *_{\mathcal{C}}(\eta_2+\theta\xi_2)]
\end{align}
where $\frac{\partial}{\partial\theta}(\theta\xi)=\xi$ for any $\xi\in\Om^*(M,Ad\,P)$.

Having defined the inner product on $\tC^*(\zeta)(M, Ad\,P)$, we  write down the Yang-Mills functional associated with the cone connection form.  Given a connection form $A$ of $Ad\,P$ over $M$ 
and a section $B\in\Om^0(M,Ad\,P)$, we write the cone connection form as  $\CA=A + \theta B$. 
Locally, if we write the usual covariant derivative as $d_A = d+ A$, then the cone covariant derivative is given by $\DC = d+\CA= d + (A + \th B)$ and the cone curvature takes the form
\begin{align*}
\CF_\CA=\DC\DC 
&= d (A + \th B) + (A + \th B) \w (A + \th B) \\
&= (F_A + \zeta B) - \th\, d_A B \,,
\end{align*}
having used the relation $\zeta=d\theta\,$.  The norm squared of the cone curvature is the cone Yang-Mills functional
\begin{align*}
S_{cY\!M}(A+\theta B)&
=\|\CF_\CA\|_{\mathcal{C}}^2
=\|(F_A+\zeta B)-\theta\, d_A B\|^2_{C}\\
&=\|F_A + \zeta B\|^2 +\| d_A B\|^2\,.
\end{align*}
The zero points of the functional gives the cone-flat condition for the pair $(A, B)$:
\begin{align}\label{cflat}
F_A+\zeta B=0,\qquad d_AB=0.
\end{align}
If satisfying \eqref{cflat}, we will call the pair $(A, B)$ cone-flat connections with respect to $\zeta$, or often, just simply cone-flat connections.

Now to obtain the equations for the critical points of the cone Yang-Mills functional, we consider the first order variation,  $(A,B) \rightarrow (A + t\eta, B + t\xi)$ where $\eta\in\Om^1(M,Ad\,P)$ and $\xi\in\Om^0(M,Ad\,P)$. For the standard curvature, we find
\begin{align*}
F_{A+t\eta}=F_A+td_A\eta+t^2 \eta \w \eta\,,
\end{align*}
where locally, $d_A \eta = d\eta + A \w \eta + \eta \w A=: d\eta + [A, \eta]$.
Furthermore, we have
\begin{align*}
S_{cY\!M}&(A+t\eta, B+t\xi)\\ 
&= \|F_{A+t\eta}+\zeta (B+t\xi)\|^2+\|d_{A+t\eta}(B+t\xi)\|^2 \\
&= \|F_A+\zeta B\|^2+\|d_AB\|^2 + 2t\left(\langle F_A+\zeta B,d_A\eta+\zeta\xi \rangle + \langle d_AB,[\eta,B]+d_A\xi \rangle \right) + o(t) \\
&= \|\CF_\CA\|_{\mathcal{C}}^2 + 2t\big( \langle d_A^*(F_A+\zeta B)+[B, d_AB],\eta \rangle + \langle \zeta^*(F_A+\zeta B)+d_A^*d_AB,\xi \rangle \big) + o(t)
\end{align*}
where for $\dim M=m$, 
\begin{align}
d_A^*=(-1)^{mk+m+1}*d_A*\,,
\end{align}
and $\zeta^*\!:\Om^k(M)\to\Om^{k-2}(M)$ is the adjoint of the $\zeta\w$ map and defined to be 
\begin{align}
\zeta^*=(-1)^{(m-k)k}*\zeta*\,.
\end{align}
Hence, the stationary solutions of the cone Yang-Mills functional satisfy
\begin{align}
d_A^*(F_A+\zeta B)+[B, d_AB] &= 0\,, \label{cYMa}\\
\zeta^*\!\left(F_A+\zeta B\right)+d_A^*d_AB &=0\,. \label{cYMb}
\end{align}
We will call a pair $(A, B)$ satisfying the above cone Yang-Mills equations \eqref{cYMa}-\eqref{cYMb} cone Yang-Mills connections.

\subsection{Elliptic property of the cone Yang-Mills solutions}
In this subsection, we shall prove the following theorem.
\begin{thm}
On a closed manifold $M$, the space of cone Yang-Mills connections modulo gauge equivalence is finite-dimensional.
\end{thm}
Our proof will follow the same line of arguments as Atiyah-Bott's proof of the analogous statement \cite{Atiyah-Bott}*{Sec. 4} for the Yang-Mills functional.
\begin{proof}
We compute the linearized  variation of the cone Yang-Mills equations.  Let  $(A,B) \rightarrow (A + t\eta, B + t\xi)$ where $\eta\in\Om^1(M,Ad\,P)$ and $\xi\in\Om^0(M,Ad\,P)\,.$  The  linearized variation of the first cone Yang-Mills equation  \eqref{cYMa} is of the form
\begin{align}
d_{A+t\eta}^*\big(&F_{A+t\eta}+\zeta(B+t\xi)\big)+[B+t\xi, d_{A+t\eta}(B+t\xi)] \nonumber\\
    &= d_A^*(F_A+\zeta B)+[B, d_AB] 
    + t\Big(  d_A^*d_A\eta- [B, [B,\eta]] +d_A^*(\zeta \xi) + [B, d_A\xi] \nonumber\\
&\qquad \qquad    + [\xi, d_A B]  + (-1)^{m+1}*[\eta, *(F_A+\zeta B)] \Big)+o(t)\,,\label{Seq}
\end{align}
and for the second equation \eqref{cYMb}, we find
\begin{align}
    \zeta^*& F_{A+t\eta}+\zeta^*\Big(\zeta(B+t\xi)\Big)+d_{A+t\eta}^*d_{A+t\eta}(B+t\xi) \nonumber\\
    &= \zeta^* F_A+\zeta^*(\zeta B)+d_A^*d_AB + t\Big( d_A^*d_A\xi+ \zeta^*(\zeta\xi) +\zeta^* d_A\eta - d_A^*[B,\eta] - *[\eta, *d_AB]  \Big)+o(t)\,.\label{Teq}
\end{align}

With \eqref{Seq}-\eqref{Teq}, we see that  $A+\theta B+t(\eta+\theta\xi)+o(t)$ describes a curve of the critical points of the cone Yang-Mills functional if and only if
\begin{align}
\Big\{d_A^*d_A\eta- [B, [B,\eta]] +d_A^*(\zeta \xi) + [B, d_A\xi]  \Big\} + [\xi, d_A B] + (-1)^{m+1}*[\eta,*(F_A+\zeta B)] &=0 \label{STsya}\\ 
\Big\{d_A^*d_A\xi+\zeta^*(\zeta\xi) +\zeta^* d_A\eta - d_A^*[B,\eta]\Big\} - *[\eta, *d_AB]&=0 \label{STsyb}
\end{align}

We can write  \eqref{STsya}-\eqref{STsyb} more simply in terms of $\DC$, the cone covariant derivative.  Locally, if $d_A = d+ A$, then $\DC = d + (A + \th B)$ and 
\begin{align*}
\DC (\eta + \th \xi) = d_A \eta + \zeta \w \xi + \th ([B, \eta] - d_A \xi)\,,
\end{align*}which we can express in matrix form as
\begin{align}\label{DCop}
\DC \begin{pmatrix} \eta \\ \xi \end{pmatrix} = \begin{pmatrix} d_A & \zeta\w \\ [B, -] & - d_A \end{pmatrix} \begin{pmatrix} \eta \\ \xi \end{pmatrix}\,,
\end{align}
and its adjoint with respect to the metric in \eqref{cmet2} by
\begin{align}\label{DCadj}
\DC^*\begin{pmatrix} \eta \\ \xi \end{pmatrix} = \begin{pmatrix} d^*_A & -[B, -]\\ \zeta^*  & - d^*_A \end{pmatrix} \begin{pmatrix} \eta \\ \xi \end{pmatrix}\,.
\end{align}
Here, we have noted that the adjoint $[B, -]^*=-[B,-]$, since for any $\gamma,\gamma'\in\Omega^k(M,Ad\,P)$, we have  
\begin{align*}
\langle [B, -]^*(\gamma), \gamma'\rangle =\langle \gamma, [B,\gamma']\rangle = \langle [\gamma, B],\gamma'\rangle = -\langle [B,\gamma],\gamma'\rangle\,,
\end{align*}
having used the adjoint-invariance property of the inner product \eqref{adjinp}.  
Now, the composition 
\begin{align*}
\DC^*\DC \begin{pmatrix} \eta \\ \xi \end{pmatrix} =
\begin{pmatrix} d^*_A d_A  -[B, [B, -] ]& d_A^*(\zeta\wedge-) + [B, d_A-] \\ \zeta^*d_A - d_A^*[B,-]   &  d^*_A d_A + \zeta^*(\zeta\wedge-) \end{pmatrix} \begin{pmatrix} \eta \\ \xi \end{pmatrix} 
\end{align*}
which reproduces exactly the terms within the curly brackets $\{\ \ldots \}$ in \eqref{STsya}-\eqref{STsyb}.  The remaining terms can be expressed as
\begin{align*}
(-1)^{m} *_\CC [\, \eta + \th \xi \,, \,*_\CC\, \FC \,] 
&=(-1)^{m} *_\CC [\,\eta + \th \xi \,,\, *_\CC(F_A + \zeta\,B - \th \, d_A B) ] \\
&=(-1)^{m} *_\CC [\eta + \th \xi\,,  \th *(F_A + \zeta\, B) -*d_A B  ]\\
&=(-1)^{m+1} *_\CC \left\{[\eta,*d_A B] + \th \left( [\xi, *d_A B] + [\eta, *(F_A + \zeta\, B)]\right)\right\}\\
&=[\xi, d_A B]+ (-1)^{m+1} *[\eta, *(F_A + \zeta\, B)] - \th \, *[\eta, *d_A B]
\end{align*}
The last equation follows from the relation $(-1)^{m+1} *[\xi, *d_A B]=[\xi, d_A B]$.  This can be seen by writing $d_A B=\sum \mu_i\otimes\alpha_i$ and $\xi= \sum \xi_i \otimes \alpha_i$ where $\mu_i\in\Omega^{1}(M)$, $\xi_i \in \Omega^0(M)$ and $\{\alpha_i\}$ is a basis of the Lie algebra $\mathfrak{g}=\Lie(G)$.  Then 
\begin{align*}
(-1)^{m+1} *[\xi, *d_A B] 
&= (-1)^{m+1}\sum *[\xi_i \otimes \alpha_i, (*\mu_j)\otimes\alpha_j] \\
&= (-1)^{m+1}\sum **(\xi_i\mu_j) \otimes [\alpha_i,\alpha_j] \\
&= \sum \xi_i\mu_j \otimes  [\alpha_i,\alpha_j] \\
&= [\xi, d_A B] \, .    
\end{align*}

In all, the linearized variation condition \eqref{STsya}-\eqref{STsyb} can be expressed concisely as 
\begin{align}\label{ldefcond}
\DC^*\DC(\eta + \th \xi) + (-1)^{m} *_\CC [\, \eta + \th \xi \,, \,*_\CC\, \FC \,] =0\,. 
\end{align}
Now, under a gauge transformation
\begin{align*}
A+\th B \rightarrow g (A + \th B)g^{-1} + g dg^{-1} =A + \th B - t (D_C\, \alpha) + o(t)
\end{align*}
having substituted on the right-hand-side  $g= e^{t\alpha}$ for $\alpha \in \mathfrak{g}$. 
To quotient out a linear variation that is a gauge transformation, we impose that the deformation satisfy the gauge-fixing condition 
\begin{align}\label{gcond}
\DC^*(\eta + \th \xi)  =0  \,.
\end{align}
Combining \eqref{ldefcond} with \eqref{gcond}, the linearized deformation $(\eta + \th \xi)$ is characterized by solving the differential system 
\begin{align}
(\DC^*\DC+\DC \DC^*)(\eta + \th \xi) + (-1)^{m} *_\CC [\, \eta + \th \xi \,, \,*_\CC\, \FC \,] =0\,.
\end{align}
This is an elliptic system since the cone Laplacian $\Delta_\CC =  \DC^*\DC+\DC \DC^*$ is elliptic.  (With \eqref{DCop}-\eqref{DCadj}, it is easily seen that the contribution to the principal symbol of $\Delta_\CC$ comes only from the standard Laplacian $d^*d+dd^*$ on the diagonal components.)  Hence, this implies that the tangent space of cone Yang-Mills connections modulo gauge equivalence is finite-dimensional.   
\end{proof}

\subsection{Comparison with Yang-Mills solutions}

Since the cone Yang-Mills functional is closely related to the Yang-Mills functional, it is natural to ask about the relationship between their critical points (i.e. solutions).  We shall study this issue starting first with some special cases.

\bigskip 

\noindent{\bf (i) Cone Yang-Mills solutions with $B=0$.}

Consider first the case of cone Yang-Mills solutions with $B=0$.  In this setting, the cone Yang-Mills equations simplify to 
\begin{align}\label{Bzcond}
d_A^*F_A= 0\,, \qquad  \zeta^*F_A=0\,.
\end{align}
Hence, cone Yang-Mills solutions with $B=0$ are a subset of Yang-Mills solutions satisfying additionally the second condition of \eqref{Bzcond}.  

In particular, when $M$ is a K\"ahler manifold and $\zeta=\om$ is the K\"ahler metric, a hermitian Yang-Mills connection that satisfies the conditions
\begin{align*}
F_A^{2,0} = F_A^{0,2}=0\,, \qquad  \om^{n-1}\w F_A =0\,,  
\end{align*}
is also a cone Yang-Mills solution satisfying \eqref{Bzcond} with $B=0$.  This is because $\om^*F_A=0$ is equivalent to $\om^{n-1}\w F_A=0$ for a K\"ahler metric.

\bigskip

\noindent{\bf (ii) Cone Yang-Mills solutions with $\zeta=0$.}

Another special case is that of setting $\zeta=0$.  Notice first for the zero point of the cone Yang-Mills functional satisfying the cone-flat condition \eqref{cflat}, the pair $(A,B)$ is cone-flat if and only if $A$ is a flat connection (i.e. $F_A=0$) and $d_A B=0$.  In general, we have the following:
\begin{lem}\label{zetaz}
Let $M$ be a closed manifold and let $\zeta=0$.  Then $(A,B)$ is a cone Yang-Mills solution if and only if $A$ is a Yang-Mills connection and $d_A B=0$.  
\end{lem}
\begin{proof}
When $\zeta=0$, $\zeta^*$ is a zero map. The second equation of cone Yang-Mills condition \eqref{cYMb} becomes just $d_A^*d_AB=0$, which implies $d_AB=0$ on a compact manifold. The first equation \eqref{cYMa} then simplifies to $d_A^*F_A=0$.  Hence, $A$ must be a Yang-Mills connection with $B$ covariantly constant.
\end{proof}

\bigskip

\noindent{\bf (iii) Cone Yang-Mills solutions when $\zeta$ is a harmonic form.}

Instead of vanishing, suppose $\zeta$ is a harmonic two-form, i.e. $d\zeta=d^*\zeta=0$.  We can obtain a similar statement to Lemma \ref{zetaz} if we require additionally that the cone Yang-Mills solution is cone-flat.
\begin{lem}\label{zetah}
Suppose $(A,B)$ is a cone-flat solution, i.e. a zero point of the cone Yang-Mills functional. If $\zeta$ is a harmonic form, then $A$ is a Yang-Mills connection.
\end{lem}

\begin{proof}
By assumption, we have $F_A=-\zeta B$ and $d_AB=0$. So
$$
d_A*F_A=-d_A(*\,\zeta B)=-(d*\zeta)B-(-1)^{m-2}(*\zeta)d_AB=0\,,
$$
implying that $A$ is a Yang-Mills connection.
\end{proof}

\begin{rmk}
When $(M^{2n}, \om)$ is a symplectic manifold and $\zeta=\om$, then a connection satisfying the curvature condition $F_A=-\om B$ such that $d_A B=0$ is called a symplectically-flat connection as introduced in \cite{TZ}.  Symplectically-flat connections are Yang-Mills connections with respect to a compatible metric.  This agrees with Lemma \ref{zetah} above since with respect to a compatible metric, $*\,\om = \om^{n-1}/(n-1)!\,$ which implies $d^*\om =0\,$, i.e. $\om$ is a harmonic form. 
\end{rmk}

\bigskip

\noindent{\bf (iv) Cone Yang-Mills solutions with connection one-form not Yang-Mills.}

Thus far, we have described special cone Yang-Mills solution pairs $(A, B)$ where the connection part $A$ is Yang-Mills.  Such is not the generic case.  Below, we shall give a simple cone-flat solution $(A, B)$ where $A$ is not Yang-Mills.  In order not to contradict Lemma \ref{zetah},  the $\zeta$ in the example below is not harmonic. 

\begin{ex}\label{cYM not YM}
Let $\Sigma$ be a Riemann surface and $A'$ a non-flat Yang-Mills connection on $\Sigma$. Let $\zeta$ be the volume form of $\Sigma$ normalized such that the total volume of $\Sigma$ is one.  Define $M$ to be the three-dimensional circle bundle $\pi:M\to\Sigma$ with Euler class given by $\zeta$.  Since $A'$ is Yang-Mills, we can write its curvature as $F_{A'}=\zeta\Phi$ such that $d_{A'}\Phi=0$. For simplicity, we will also use $\zeta$, $A'$ and $\Phi$ to denote their pullbacks on $M$.

Let $\theta\in\Om^1(M)$ be the global angular one-form of the circle bundle $M$, i.e. $d\theta=\zeta$, and also let $c\in \mathbb{R}$.  For $A=A'+c\,\theta\Phi$ and $B=-(1+c)\Phi$ on the pullback bundle $\pi^*(Ad\,P)$, we find
$$F_A=F_{A'}+c\,\zeta\Phi-c\,\theta d_{A'}\Phi = (1+c)\zeta\Phi = -\zeta B\, ,$$
and
$$d_A B = -(1+c)d_{A'}\Phi-c(1+c)\theta[\Phi,\Phi]=0\, ,$$
that is, $(A, B)$ is a cone-flat solution.  However, the connection form $A$
is not Yang-Mills in general with respect to the volume form on the circle bundle $M$, $dvol_M=\zeta\w \theta$, since
$$d^*_A F_A=*\;d_A *F_A=-*d_A(\theta B)=-*\zeta B=(1+c)\theta\Phi$$
which is non-zero unless $c=-1$.  (We have assumed $A'$ is not a flat connection, and therefore, $F_{A'}=\zeta \Phi\neq 0$.)  Hence, generally, for any $c\neq -1$, $A$ is not a Yang-Mills connection.
\end{ex}

\bigskip

\noindent{\bf (v) Yang-Mills connections that can not be a part of a cone Yang-Mills solution.}

For a cone Yang-Mills solution, $(A,B)$, we have seen that $A$ need not be a Yang-Mills connection.  In the reverse direction, we can ask if given a Yang-Mills connection $A$, will there always exist a $B$ such that $(A,B)$ is a cone Yang-Mills solution?  The answer is no, as is shown in the example below.

\begin{ex}
Let $M=T^4=\mathbb{R}^4/2\pi\mathbb{Z}^4$ be the 4-torus described by identifications $x_i\sim x_i+2\pi n_i$ for $i=1,2,3,4$ and $n_i \in\mathbb{Z}$.  Let $\zeta=dx_1\wedge dx_2+dx_3\wedge dx_4$. We will take the Riemannian metric to be 
$$
g=(dx_1)^2+(dx_2)^2+\frac{1}{f}(dx_3)^2+f\,(dx_4)^2
$$
where $f=\dfrac{3+2\sin 2x_2\cos x_3}{1-\frac{1}{2}\sin 2x_2\cos x_3}$. 

Consider a principal $U(1)$ bundle over $M$ with the circle coordinate identified by
\begin{align*}
y\sim y+2\pi n_5-n_1x_3\,
\end{align*}
for $n_5\in\mathbb{Z}$.  The global connection one-form on the $U(1)$ bundle can be taken to be
$$
A=dy+\frac{1}{2\pi}x_1\,dx_3+\frac{1}{4\pi}\sin 2x_2\sin x_3 \,dx_1.
$$
The curvature two-form is then given by $$
F_A=dA=-\frac{1}{2\pi}\cos 2x_2\sin x_3 \;dx_1\w dx_2+\frac{1}{2\pi}(1-\frac{1}{2}\sin 2x_2\cos x_3)dx_1\w dx_3\,.
$$
Moreover, 
\begin{align*}
    d*F_A &= \frac{1}{2\pi}d\left[-\cos 2x_2\sin x_3\; dx_3\w dx_4-(1-\frac{1}{2}\sin 2x_2\cos x_3)f\;dx_2\w dx_4\right] \\
    &= \frac{1}{2\pi}d\left[-\cos 2x_2\sin x_3\; dx_3\w dx_4-(3+2\sin 2x_2\cos x_3)dx_2\w dx_4\right] = 0\,.
\end{align*}
Thus, $F_A$ satisfies the abelian Yang-Mills equation.  However, there does not exist a function $B$ on $T^4$ such that $(A,B)$ satisfy the cone Yang-Mills equations \eqref{cYMa}-\eqref{cYMb}.  Equation \eqref{cYMa} implies in the abelian case  $d^*(F_A+\zeta B)=0\,$. But since we know already $d^*F_A=0\,$, this means $B$ is required to satisfy
\begin{align} \label{bquest}
0=d*(\zeta B)=d(\zeta B)=\zeta\wedge dB\,.
\end{align}
With $\zeta$ being a non-degenerate two-form, \eqref{bquest} gives the condition $dB=0$.

However, imposing the second cone Yang-Mills equation \eqref{cYMb} results in a contradiction.  For if $dB=0$, \eqref{cYMb} reduces to the condition $\zeta^*F_A+\zeta^*\zeta B=0$. This implies in particular
\begin{align}\label{Bquest2}
B = \frac{1}{2}\zeta^*\zeta B = -\frac{1}{2}\zeta^*F_A = \frac{1}{4\pi}\cos 2x_2\sin x_3\,,
\end{align}
and clearly, $dB\neq 0$, which contradicts the condition from \eqref{bquest}.  Hence, there is no $B$ that can satisfy the cone Yang-Mills equations.
\end{ex}

\subsection{Uniqueness of cone Yang-Mills solutions for a fixed connection}

In general, given an  arbitrary connection $A$, there does not exist a $B$ that would make $(A,B)$ a cone Yang-Mills solution.  To illustrate this in a simple example, let $M$ be a closed manifold and $\zeta=0\,$.   Then by Lemma \ref{zetaz}, we know that  $A$ must be a Yang-Mills connection and $d_A B =0\,$.  Hence, there exists no $B$ that would give a cone Yang-Mills solution if $A$ is not Yang-Mills.

However, it is interesting to ask that given a cone Yang-Mills solution $(A,B)$, how many different $B$'s with $A$ fixed would also be a cone Yang-Mills solution? In the case of $\zeta=0$ and $M$ closed, if $(A,B)$ is a cone Yang-Mills solution, then so are all $(A,B+\Phi)$ with $d_A\Phi=0$.
For $\zeta\neq 0$, we are able to obtain results in two special cases.  First, in two dimensions and suppose $A$ is a Yang-Mills connection, we have the following:
\begin{prop}
    Suppose $M$ is a closed Riemann surface and $\zeta$ is its volume form. For each Yang-Mills connection $A$, there exists a unique $B$ such that $(A,B)$ is a cone Yang-Mills solution. In fact, such a pair $(A,B)$ is always cone-flat.
\end{prop}
\begin{proof}
    Since $A$ is a Yang-Mills connection on a Riemann surface, $F_A=\zeta \Phi$ with $d_A\Phi=0$. So we can choose $B=-\Phi$ and then $(A,B)$ is a cone-flat solution.  We will show below that this is the unique cone Yang-Mills solution for a fixed $A$, a Yang-Mills connection.
        
Generally, suppose $(A,B)$ is a cone Yang-Mills solution. As $\zeta$ is a volume form, the condition $\zeta^*(F_A+\zeta B)+d_A^*d_AB=0$ becomes
    \begin{align}\label{2deq}
        \Phi+B+d_A^*d_A B = 0\,.
    \end{align}
Let $d_A$ act on both sides.  This implies $d_A B=-d_Ad_A^*d_A B$, and therefore,
\begin{align*}
\langle d_AB,d_Ad_A^*d_A B \rangle = -\|d_AB\|^2 \leq 0\,.
\end{align*}
On the other hand,
\begin{align*}
    \langle d_AB,d_Ad_A^*d_A B \rangle = \langle d_A^*d_AB,d_A^*d_A B \rangle = \|d_A^*d_AB\|^2\geq 0.
\end{align*}
So $\langle d_AB, d_Ad_A^*d_A B \rangle$ has to vanish. It follows that $\|d_AB\|^2=0$, which implies, $d_AB=0$, and by \eqref{2deq}, $B=-\Phi$. So such $B$ is unique when $A$ is Yang-Mills.
\end{proof}

Another special case where we can constrain $B$ is when $\zeta$ is the symplectic form on a symplectic manifold. If the structure group is abelian, then $B$ must be unique.

\begin{prop}
Suppose $A$ is a connection on the associated adjoint bundle $Ad\,P$ over a symplectic manifold $(M^{2n},\omega)$ and the Riemannian metric is compatible with $\om$. Take $\zeta=\omega$. Then there is at most one $B$ that satisfies $[B, d_A B]=0$ and such that $(A,B)$ is a cone Yang-Mills solution. In particular, if the structure group is abelian, there is at most one cone Yang-Mills solution pair $(A,B)$ for any given connection $A$.
\end{prop}
\begin{proof}
Let $\dim M=2n$. With respect to a compatible metric, $*\,\zeta= *\,\om=\frac{1}{(n-1)!}\om^{n-1}$, and $\zeta^*\zeta=n$. 

Suppose both $(A, B_1)$ and $(A, B_2)$ are cone Yang-Mills solutions.  Assume also $[B_1, d_AB_1]=[B_2, d_AB_2]=0$ which is identically true when the structure group is abelian. Then \eqref{cYMa}-\eqref{cYMb} imply for $B_1$,
\begin{align*}
-*d_A(*F_A+\frac{1}{(n-1)!}\om^{n-1} B_1) = 0\,, \\
\zeta^* F_A+n B_1+d_A^*d_AB_1 =0\,,
\end{align*}
and $B_2$ satisfies identical equations. So by the first equation, we have
$$
\om^{n-1}\w d_A B_1=\om^{n-1}\w d_A B_2.
$$
It follows that $d_AB_1=d_AB_2$ because $\om^{n-1}:\Omega^1(M)\to\Omega^{2n-1}(M)$ is an isomorphism. With this, the second equation becomes 
$$
n(B_1-B_2)=d_A^*d_A(B_2-B_1)=0\,,
$$
which implies, $B_1=B_2$.
\end{proof}

The above proposition avoids considering the term $[B, d_AB]$. If $[B, d_AB]=0$ holds true for all cone Yang-Mills solutions $(A,B)$, then we would have obtained the uniqueness of $B$ when $\zeta$ is the symplectic structure. But as we shall see in Example \ref{TaubN} in the next section, a cone Yang-Mills solution in general need not satisfy $[B, d_AB]= 0$.

\section{Special solutions of the cone Yang-Mills functional}

\subsection{Two-dimensional solutions with abelian gauge group}

Consider cone Yang-Mills solutions on a closed Riemann surface with abelian structure group. In this case, the Euler-Lagrange equations reduce to
\begin{align}
d^*(F_A+\zeta B)&= 0\,, \label{2da1}\\
\zeta^*\!\left(F_A+\zeta B\right)+d^*dB &=0\,. \label{2da2}
\end{align}
The first equation \eqref{2da1} implies that $*(F_A+\zeta B)$ is a constant $c$. So we can assume that 
\begin{align}\label{FAbcw}
F_A+\zeta B=c\,\omega\,,
\end{align} 
where $\omega$ is the volume form of $M$. By Hodge decomposition, we can write 
$$\zeta=c'\omega+dd^*(f\omega)=c'\omega-d*df\,,$$ 
where $c'$ is a constant and $f$ is a function on $M$.

For any function $\phi$ on $M$, we have
\begin{align*}
    \langle \phi,\zeta^*\!\left(F_A+\zeta B\right) \rangle & = \langle \zeta \phi, \left(F_A+\zeta B\right) \rangle = \langle \phi(c'\omega-d*df),c\,\omega \rangle \\
    &= \int_M c\phi(c'\omega-d*df) \\
    &= \int_M cc'\phi\,\omega-\int_M d(c\phi *df)+\int_M c\,d\phi\wedge *\,df \\
    &= \langle \phi,cc'\rangle+\langle d\phi, c\,df \rangle  \\
    &= \langle \phi, cc'+c\,d^*df\rangle
\end{align*}
and therefore,
\begin{align}\label{xisimp}
\zeta^*\!\left(F_A+\zeta B\right) = cc'+c\,d^*df\,.
\end{align} 
Plugging this into \eqref{2da2}, we find that $cc'$ is a $d^*$-exact constant; hence, $cc'$ must be zero.

If $c'\neq 0$, then $c$ must vanish and then $F_A+\zeta B=0$. By \eqref{2da2}, $d^*dB=0$ which implies $dB=0$ since $M$ is closed. Thus, we have obtained the following statement.

\begin{prop}
    Suppose $M$ is a closed Riemann surface, $\zeta$ is a non-exact two-form on $M$, and the structure group is abelian. Then all cone Yang-Mills solutions are cone-flat.
\end{prop}

Now we consider the case $c'=0$, that is,  $\zeta=-d*df$ is a $d$-exact form.  Together, \eqref{2da2} and \eqref{xisimp} imply $d^*d(cf+B)=0$, and so $cf+B$ is a constant. We can thus write $B=-cf+c''$ for some constant $c''$. Hence, by \eqref{FAbcw}, $F_A=c\,\omega-\zeta B=c\,\omega+(c''-cf)d*df$. Therefore, we find that the constants $c$ and $c''$ parametrize the cone Yang-Mills solutions. And finally, in the special case where $f=0$, i.e. $\zeta$ vanishes, the critical points must satisfy $F_A=c\,\om$ and $B=c''$.

\subsection{Three-dimensional solutions from duality relations}\label{3dsol}

Recall in four dimensions, there are special Yang-Mills solutions that satisfy the first-order self-dual/anti-self-dual conditions, $*\,F_A = \pm\, F_A\,$. The intuition that cone Yang-Mills functional can be interpreted as a dimensional reduction of the Yang-Mills functional suggests an analogous duality condition $*_{\mathcal{C}}\,\CF_\CA=\pm\,\CF_\CA$ over 3-manifolds, where $\CF_\CA=(F_A+\zeta B)-\theta d_A B$. With the $*_{\mathcal{C}}$ acting on $\tC^k(\zeta)$ given by \eqref{conestar}, we have  
\begin{align*}
*_{\mathcal{C}}\,\CF_\CA &= *_{\mathcal{C}}\left[(F_A +\zeta B) - \theta d_A B\right]\\
& = * (-d_A B) + \theta \left[* (F_A + \zeta B)\right].
\end{align*}
Hence, we find that $*_{\mathcal{C}}\,\CF_\CA=\mp\,\CF_\CA$ implies
\begin{align}\label{3duality}
F_A + \zeta B = \pm * d_A B
\end{align}
Note that when $\zeta=0$, the condition becomes the Bogomolny monopole equations.  

We will show that a solution of  
\eqref{3duality} is automatically a solution of the cone Yang-Mills equations in two different ways.  First, we can check directly that  \eqref{3duality} implies the cone Yang-Mills equations \eqref{cYMa}-\eqref{cYMb}.  Applying $d^*_A$ to  \eqref{3duality}, we find the following:
\begin{align*}
d_A^*(F_A+\zeta B) &= \pm\, d_A^* *d_A B\\
&= \pm * d_A d_A B = \pm *[F_A , B] \\
& = [(d_A B - *\,\zeta B), B] = [d_A B, B]
\end{align*}
where we have used the three-dimensional relations $** =1\,$, and  $d_A^* = (-1)^k* d_A\, *$ acting on a $k$-form.  The above implies the first cone Yang-Mills equation \eqref{cYMa}, $d_A^*(F_A+\zeta B)+[B, d_AB] = 0\,$.  Furthermore, it also follows from \eqref{3duality} that
\begin{align*}
d_A^* d_A B & = - * d_A (*\,d_A B) = \mp * d_A (F_A + \zeta B)\\
&= \mp *\zeta\w  d_A B = \mp *\zeta\, * (*d_A B)\\
&= -\zeta^*(F_A + \zeta B)
\end{align*}
which implies the second cone Yang-Mills equation \eqref{cYMb}, 
$\zeta^*\!\left(F_A+\zeta B\right)+d_A^*d_AB =0\,$.

In the second method, analogous to the standard Yang-Mills instanton argument, we can express the three-dimensional cone Yang-Mills functional in the following manner:
\begin{align}\label{3bound}
\|\CF_\CA\|_{\mathcal{C}}^2 &=  \int_M c_G\,\tr [(F_A+\zeta B)\wedge *(F_A+\zeta B) \,+\, d_A B\wedge *\,d_A B)] \nonumber\\
&= \int_M c_G \,\tr \left[(F_A+\zeta B\mp*\,d_A B)\wedge *(F_A+\zeta B\mp*\,d_A B)\right] \;\pm\; 2\int_M c_G\, \tr \left[(F_A+\zeta B)\wedge d_A B\right] \nonumber\\    
&\geq \pm \:2\int_M c_G \, \tr \left[(F_A+\zeta B)\wedge d_A B\right].
\end{align}
The equality holds only when $F_A+\zeta B= \pm *d_A B$.   Importantly, the bounding integral is a boundary term  
\begin{align}
Q&=\int_M c_G\,\tr \left[(F_A+\zeta B)\w d_A B\right] = \int_M c_G\,\tr \left[d\left(F_A\, B + \dfrac{1}{2}\, \zeta\, B^2\right) - (d_A F_A) B + \zeta\w B[A,B]\right] \nonumber\\
&=\int_{\partial M} c_G\,\tr \left[F_A\, B + \dfrac{1}{2}\,\zeta\, B^2\right]\label{mcharge}
\end{align}
which is the action of a two-dimensional topological $BF$-type theory.  We note that any infinitesimal local variation of $(A,B)$ away from the boundary does not affect the bound.  Hence, $F_A+\zeta B= \pm *d_A B$ must be a critical point of the cone Yang-Mills functional. 

The duality-type condition of \eqref{3duality} helps simplify the search for cone Yang-Mills solutions in three dimensions. Notably, it is first-order and hence more tractable compared with the second-order cone Yang-Mills equations \eqref{cYMa}-\eqref{cYMb}.  As mentioned, when $\zeta=0$, the condition becomes the Bogomolny equations and then the known three-dimensional Bogomolny-Prasad-Sommerfeld (BPS) monopole solutions are trivially cone Yang-Mills solutions with $Q$ of \eqref{mcharge} being proportional to the magnetic monopole charge (for a review, see \cite{Prasad}).
More generally, for $\zeta\neq 0$, we can look for self-dual/anti-self-dual Yang-Mills instanton solutions in four dimensions on spaces that can be described as a circle bundle over a three-manifold.  If the four-dimensional Yang-Mills instanton solutions are invariant under the $S^1$ circle action, then we can dimensionally reduce over the circle and obtain solutions that satisfy \eqref{3duality}.  We give such an example below coming from the Taub-NUT gravitational instanton.

\begin{ex}\label{TaubN}
The Taub-NUT gravitational instanton solution can be thought of as a self-dual Yang-Mills solutions of a tangent bundle with structure group $SU(2)\subset SO(4)$.  With a point removed, the four-dimensional space can be considered as a circle bundle, $S^1\to X \to \mathbb{R}^3 - \{0\}$ (see, for example, the description in  \cite{LeBrun}).   Dimensionally reducing this solution leads to a non-abelian cone Yang-Mills solution on $M=\mathbb{R}^3 - \{0\}$ satisfying \eqref{3duality}.

We start with the Taub-NUT metric written in Gibbons-Hawking form:
\begin{align}\label{TNmetric}
ds_{TN}^2 = e^{2\phi}\left((dx^1)^2 + (dx^2)^2 + (dx^3)^2\right) + e^{-2\phi}\theta^2 \,,  
\end{align}
where $e^{2\phi}=1+2/r$ is a positive function on $\mathbb{R}^3 - \{0\}$ with $r^2=(x^1)^2+(x^2)^2+(x^3)^2$ 
and  the $d$-closed two-form $\zeta$ is defined to be
\begin{align}
\zeta &= d\theta = \pm\, *_0d( e^{2\phi}) \nonumber\\
&
= \pm\, \epsilon_{ijk}\,\pa_k\phi \, e^{2\phi}\; dx^i\w dx^j \label{zetadef}
\end{align} 
with $*_0$  being the Euclidean Hodge star on $M=\mathbb{R}^3-\{0\}$.   The requirement that $d\zeta =0$
gives the following condition on $\phi$ on $\mathbb{R}^3-\{0\}$:
\begin{align}\label{phicond}
\sum_{k=1}^{3} \left[ \partial^2_k\phi + 2 (\partial_k \phi)^2 \right] = 0 \,.
\end{align}

For the Taub-NUT metric in \eqref{TNmetric}, we have the following basis of moving frame (i.e. orthonormal frame) of 1-forms 
\begin{align*}
e^i = e^\phi dx^i \quad i=1,2,3\,, \quad \text{ and} \qquad e^4= e^{-\phi}\theta\,.
\end{align*}
These result in the following connection 1-forms (for a reference, see \cite{Oh})
\begin{align}\label{ofc1}
\om^i{}_j&= -\om^j{}_i= e^{-\phi}\left(-\pa_i \phi \,e^j + \pa_j\phi \,e^i \mp\, \epsilon_{ijk}\pa_k\phi\, e^4\right)\\\label{ofc2}
\om^i{}_4& = -\om^4{}_i= e^{-\phi}\left(\pm\epsilon_{ijk} \pa_j\phi \,e^k+ \pa_i \phi \,e^4\right)
\end{align}
which satisfy the torsionless condition $de^r + \om^r{}_s \w e^s=0$.  (Regarding indices, we will let the indices $i,j, k,l,p,q$ take values in the set $\{1,2,3\}$ and $r,s,t,u$ take values in the set $\{1,2,3,4\}$.)  Notice that these connection 1-forms are also anti-self-dual/self-dual in the sense that
\begin{align*}
\om_{rs} = \mp\, \dfrac{1}{2}\,\epsilon_{rstu}\,\om_{tu}\,.
\end{align*}
Hence, the structure group of the bundle reduces to an $SU(2)$ subgroup of $SO(4)$.  It can be checked that the resulting Taub-NUT curvature two-form $R^r{}_s = d\om^r{}_s + \om^r{}_t \w \om^t{}_s$ is correspondingly anti-self-dual/self-dual.

Now to perform a dimensional reduction,
we take the metric on $X$ to be 
\begin{align}
ds_X^2=e^{2\phi}ds_{TN}^2 = e^{4\phi}\left((dx^1)^2 + (dx^2)^2 + (dx^3)^2\right) + \theta^2 \,,
\end{align}
conformally scaled by the $e^{2\phi}$ factor in order to match the desired metric form of $ds^2_X = ds^2_M + \theta^2$ as in \eqref{Xmetric}.  We note that the four-dimensional Yang-Mills anti-self-dual/self-dual curvature equation is invariant under an overall conformal scaling of the metric.  Therefore, for the cone pair $(A, B)$ on $M=\mathbb{R}^3 - \{0\}$, we can just read off from \eqref{ofc1}-\eqref{ofc2} by setting $\CA^r{}_s = (A + \theta B)^r{}_s=\om^r{}_s\,$.  We find
\begin{align}
A^i{}_j &= -A^j{}_i= -\pa_i\phi\, dx^i + \pa_j\phi\, dx^j\,, \ \ \qquad B^i{}_j = -B^j{}_i=\mp\,\epsilon_{ijk} e^{-2\phi}\pa_k\phi\,,\label{nAsol1}\\
A^i{}_4&=-A^4{}_i=\pm\,\epsilon_{ijk}\pa_j\phi dx^k\,,\qquad\qquad\quad B^i{}_4=-B^4{}_i= e^{-2\phi}\pa_i\phi\,.\label{nAsol2}
\end{align}
These lead to the following:
\begin{align*}
(F_A)^i{}_j&= -\pa_i\pa_k\phi \;dx^k \w dx^j + \pa_j\pa_k\phi \;dx^k \w dx^i - \epsilon_{ikl}\epsilon_{jpq}\,\pa_k\phi\pa_p\phi\; dx^l\w dx^q \\ & \qquad + \pa_i\phi\pa_k\phi \;dx^k\w dx^j - \pa_j\phi\pa_k\phi\; dx^k \w dx^i - (\pa_k\phi)^2 dx^i\w dx^j\,, \\
(F_A)^i{}_4&=\pm\left( \pa_i\phi\,\epsilon_{jkl}\pa_j\phi\; dx^k\w dx^l - \epsilon_{ijk} \pa_j\pa_l\phi\; dx^k\w dx^l \right) ,\\
(d_A B)^i{}_j&= \pm\, e^{-2\phi}\left[2\, \epsilon_{ijk}\,\pa_k\phi\pa_l\phi\; dx^l - \epsilon_{ijk}\pa_k\pa_l\phi\;dx^l +2 \left(\pa_i\phi\,\epsilon_{jkl}-\pa_j\phi\,\epsilon_{ikl}\right)\pa_k\phi\;dx^l  \right],\\
(d_A B)^i{}_4&=e^{-2\phi}\left(-4\,\pa_i\phi\pa_j\phi\;dx^j + 2(\pa_j\phi)^2 dx^i + \pa_i\pa_j\phi\;dx^j\right).
\end{align*}
With \eqref{zetadef}-\eqref{phicond}, it can be straightforwardly checked that the above solution satisfies $F_A+\zeta B = \pm\, * d_A B$ where the Hodge star is defined with respect to the three-dimensional metric
\begin{align*}
ds^2_M = e^{4\phi}\left((dx^1)^2 + (dx^2)^2 + (dx^3)^2\right)\,.
\end{align*} 
Hence, $(A, B)$ as defined in \eqref{nAsol1}-\eqref{nAsol2} gives us a highly non-trivial, non-abelian solution of the cone Yang-Mills equations.  Moreover, for this solution, it can be checked that $[B, d_A B]\neq 0$.  

\end{ex}

\smallskip

\begin{rmk}
We can consider dimensionally reducing the three-dimensional duality condition $*_{\mathcal{C}}\,\CF_\CA=\mp\,\CF_\CA$ over another circle down to a two-dimensional manifold $N$. For the metric, we will assume $ds_M^2 = ds_N^2 + \chi^2$
where $\chi$ is the global connection one-form (i.e.~the global angular form) of the circle bundle $M$ over $N$.  The dimensionally reduced three-dimensional connection can then be expressed as $A + \chi b$ where $A$ is now a connection form over $N$ and  $b\in\Om^0(N,Ad\,P)$.  The three-dimensional cone curvature then takes the form
\begin{align*}
    \CF_\CA = F_{A+\chi b}+\zeta B-\theta d_{A+\chi b} B &= F_A+\zeta B+ (d\chi) b -\chi d_A b-\theta d_A B-\theta \chi [b,B]\,,
\end{align*}
with 
\begin{align*}
*_{\mathcal{C}}\, \CF_\CA = \theta \chi \left[*\left(F_A + \zeta B + (d\chi) b\right)\right] - \theta * d_A b + \chi * d_A B - * [b, B]\,.
\end{align*}
The condition $*_{\mathcal{C}}\,\CF_\CA=\mp\, \CF_\CA$ then implies the following two equations:
\begin{align}\label{2Dab}
    F_A+\zeta B + (d\chi)b &= \pm * [b,B]\,,\\
    *\,d_A B &= \pm\, d_A b \label{2Dabb}\,.
\end{align}
In the case where the circle is trivially fibered over $N$, then we can set $d\chi=0$.  Moreover, when $\zeta=d\chi=0\,$, the above equations become equivalent to Hitchin's equations.
\end{rmk}

\section{Classification of cone-flat bundles with respect to a non-degenerate two-form}

In this section, we study the question what bundles can carry a pair $(A, B)$ satisfying the cone-flat condition:
\begin{align}\label{cflat1}
F_A+\zeta B=0\,,\qquad d_AB=0\,.    
\end{align}

Now if $\zeta=0$, the cone-flat condition reduces to (1) $F_A=0$, i.e. $A$ is a flat connection, and (2) $d_A B=0$, that is, $B$ is a covariantly constant section.  Since we can always set $B=0$, the $\zeta=0$ cone-flat bundles are just the usual flat bundles, and flat bundles are well-known to be classified by the conjugacy classes of homomorphisms from $\pi_1(M)$ to the structure group $G$ of the fiber bundle (see, for example \cite{Morita}).  As the classification for the $\zeta=0$ case is well-understood, we will in the following always assume that $\zeta\neq 0$.  

To simplify our consideration, we will additionally assume that $\zeta$ is a non-degenerate two-form.  Being $d$-closed and non-degenerate, $\zeta$ is then a symplectic structure and $M$ must be even-dimensional.  Additionally, $\zeta$ being non-degenerate leads to two simplifications in considering cone-flat pairs $(A, B)$.  First, non-degeneracy ensures that $\zeta$ nowhere vanishes, and hence, the first cone-flat equation $F_A=-\zeta B$, implies $B$ is determined by the connection $A$.    Second, the second cone-flat equation $d_AB=0$ is automatically satisfied and redundant when $\dim M \geq 4$.  This is because the first cone-flat equation together with the Bianchi identity imply
\begin{align} \label{BIcf}
d_A F_A = - \zeta  \w d_A B =0\,.
\end{align}  
If $\zeta$ is non-degenerate and $\dim M\geq 4$, then \eqref{BIcf} implies $d_A B = 0\,$.  In all, $\zeta$ is non-degenerate simplifies our consideration of cone-flat pairs $(A, B)$ to checking only that that the curvature is proportional to $\zeta$, i.e. $F_A=-\zeta \, B$,  and only if $\dim M = 2$, we need to also check that the resulting $B$ is covariantly constant.

When $M$ is a two-dimensional Riemann surface, $\zeta$ being non-degenerate implies that it is the product of a nowhere vanishing function times the volume form.  If we take $\zeta$ to be the volume form, then we saw in \eqref{2DYM} that the cone-flat condition \eqref{cflat1} becomes equivalent to  the Yang-Mills equation for $A$ with $B=-*\, F_A$. Concerning Yang-Mills connections in two dimensions, Atiyah-Bott \cite{Atiyah-Bott}*{Theorem 6.7} gave a classification of principal bundles over Riemann surfaces that carry a Yang-Mills connection.  So we ask whether Atiyah-Bott's classification of Yang-Mills bundles in two dimensions can be extended to cone-flat bundles in $\dim M =2n \geq 2$ with respect to a closed, non-degenerate two-form $\zeta$. 
Indeed, we obtain the following for principal bundles which generalizes  Atiyah-Bott's result.  

\begin{thm}\label{classification}
Let $M$ be a path connected manifold, $\zeta\in\Omega^2(M)$ be a non-zero, non-degenerate, closed two-form, and $G$ be a Lie group. There exists a bijective correspondence between the following sets:
\begin{align*}
\left\{\begin{matrix}\text{isomorphism classes of cone-flat connections}
\\ \text{with respect to $\zeta$ on  $G$-bundles over $M$}
\end{matrix}\right\}
\simeq
\left\{\begin{matrix}\text{conjugacy classes of}\\ \text{homomorphisms }
\rho: \Gamma\to G\end{matrix}\right\},
\end{align*}
where $\Gamma$ is an $\mathbb{R}/\overbar{H}$ extension of $\pi_1(M)$ with $\overbar{H}\subset\mathbb{R}$ being the closure of the group 
\begin{align}\label{Hdef}
H:=\left\{ \int_\mathcal{S}\zeta \ \big| \  \mathcal{S} \text{ is a representative in } \pi_2(M) \right\}.
\end{align}
\end{thm}

\

As in the theorem, we will assume in the remainder of this section that $M$ is path-connected.  We can give a more explicit description of $\Gamma$ in terms of path spaces.  To do so, we first introduce some notations.

Denote by $\Psi(M)$ the path space consisting of equivalence classes of closed, piecewise smooth paths on $M$.   Two paths  $\alpha_1, \alpha_2:[0,1]\to M$ are considered to be equivalent in $\Psi(M)$ if they have the same image and orientation.   
Specifically, $\alpha_1$ and $\alpha_2$ are equivalent if there exists a piecewise smooth increasing function $\phi:[0,1]\to [0,1]$ with $\phi(0)=0$ and $\phi(1)=1$ such that $\alpha_2=\alpha_1\circ\phi$.

For two paths $\alpha_1,\alpha_2:[0,1]\to M$ such that $\alpha_1(1)=\alpha_2(0)$, we define their multiplication by concatenation:  
\begin{align*}
(\alpha_2\alpha_1)(s)=
\begin{cases}
\alpha_1(2s), & s\in [0,\frac{1}{2}]\,,\\
\alpha_2(2s-1), & s \in [\frac{1}{2},1]\,.
\end{cases}
\end{align*}
A connected subpath of a path is also a path.  We define  
\begin{align*}
\alpha_{(t)}(s)= \alpha(s\,t)\,, \qquad t\in [0,1]\,.
\end{align*}
which is a one-parameter family of subpaths, parametrized by $t$, all with the same starting point $\alpha(0)$, and ending point being at $\alpha(t)$.

The following subsets of $\Psi(M)$ will be of interest:
\begin{itemize}
\item $\Psi^b_a(M)$ be the space of paths from a base point $a\in M$ to another point $b\in M$. 
\item $\Psi^a_a(M)$ be the semigroup in $\Psi(M)$ that consists of loops with base point $a$.   
\item $\Psi^a_{a,0}(M)$ be the subsemigroup of $\Psi^a_a(M)$ that consists of contractible loops.
\item $\Psi^a_{a,\zeta}(M)$ be the subsemigroup of $\Psi^a_{a,0}(M)$ that consists of loops which are the boundary of some disk (contractible 2-chain with proper orientation) $D$ in $M$ satisfying $\int_D\zeta=0$. 
\end{itemize}
Quotienting them leads to the following spaces:
\begin{enumerate}
\item $\Psi^a_a(M)/\Psi^a_{a,0}(M)=\pi_1(M)$ is the fundamental group where the identity element is the equivalence class of the constant path at $a$, and the inverse reverses the path orientation, i.e. $\alpha^{-1}(s)=\alpha(1-s)\,$.
\item $\Psi^a_{a,0}(M)/\Psi^a_{a,\zeta}(M)\cong\mathbb{R}/{H}$.  As $\zeta\neq 0$, for any real number $c$ there exists a contractible loop $\gamma$ that is the boundary of some disk $D$ such that $\int_D\zeta=c$. So $\Psi^a_{a,0}(M)/\Psi^a_{a,\zeta}(M)$ is equivalent to a quotient group $\mathbb{R}/{H}$ by the identification of a loop $\gamma\mapsto [\,\int_D \zeta\,] $, where 
$$H=\left\{ \int_\mathcal{S} \zeta~ \big|~ \mathcal{S} \text{ is a representative in } \pi_2(M) \right\}.$$
The group $H$ comes from noting that it is possible that the sum of two contractible curves  $\partial D_1 + \partial D_2$ can come from the vanishing boundary of a two-sphere formed by two hemisphere disks, $D_1$ and $D_2$, glued together at the equator such that $\partial D_1 = - \partial D_2$.    
\item $\Psi^a_a(M)/\Psi^a_{a,\zeta}(M)=:\!\!\Gamma\,$.  This is the extension of $\pi_1(M)$ by $\mathbb{R}/\overline{H}$ that appears in the statement of the theorem. In the case where $H$ is dense and thus $\mathbb{R}/H$ is not a Lie group,  $\Gamma=\pi_1(M)$.
\end{enumerate}

Explicitly, there are three possibilities for $H$. Let $H^+$ be the subset of $H$ with positive numbers.

\textbf{Case 1.}
When $H^+$ is empty, $H=0$ and $\mathbb{R}/H=\mathbb{R}$. In this case, $\Gamma=\Psi^a_a(M)/\Psi^a_{a,\zeta}(M)$ is an $\mathbb{R}$-extension of $\pi_1(M)$. (For instance, this occurs for any Riemann surface $\Sigma_g$ with genus $g\geq 1$, since then $\pi_2(\Sigma_g)=0\,$.)

\textbf{Case 2.}
When $H^+$ has a minimal number, $H\simeq\mathbb{Z}$ and $\mathbb{R}/H \simeq S^1$. In this case $\Gamma=\Psi^a_a(M)/\Psi^a_{a,\zeta}(M)$ is an $S^1$-extension of $\pi_1(M)$. (This occurs for Riemann surface of genus $g=0$ and $\zeta$ is not an exact form.)

\textbf{Case 3.}
When $H^+$ is non-empty and has no minimal number, $H$ is dense in $\mathbb{R}$ and $\mathbb{R}/H$ is not a Lie group. In this case, $\Gamma=\pi_1(M)$. This case is impossible when $c\,\zeta$ is an integral cohomology class for some number $c\neq0$.

We now proceed to prove the classification theorem Theorem \ref{classification}.   Our proof will be similar to that given by Morrison \cite{Morrison} for Yang-Mills bundles in the two-dimensional case.  

We will first consider the case that $\mathbb{R}/H$ is a Lie group. By the discussion above, $\Gamma=\Psi^a_a(M)/\Psi^a_{a,\zeta}(M)$ is either an $\mathbb{R}$ or an $S^1$ extension of $\pi_1(M)$. 

\bigskip

\noindent\textbf{Proof for Case 1 and 2: $\mathbb{R}/H=\mathbb{R}$ or $S^1$.}

\medskip

\noindent\textbf{Step 1.} Given a homomorphism $\rho:\Psi^a_a(M)/\Psi^a_{a,\zeta}(M)\to G$, construct a corresponding principal $G$-bundle $P_{\rho}$ with a connection $A_{\rho}\,$.

Let $\Psi_a(M)$ be the space of classes of paths starting at $a\in M$, where any two paths $\delta_1,\delta_2$ are identified if $\delta_1(1)=\delta_2(1)$ and $(\delta_2)^{-1}\delta_1\in\Psi^a_{a,\zeta}(M)$. Notice that $\Psi_a(M)$ is a principal bundle over $M$. The projection to the base space $M$, $\tau:\Psi_a(M)\to M$, is given by $\delta\mapsto\delta(1)$, and its structure group is $\Gamma=\Psi^a_a(M)/\Psi^a_{a,\zeta}(M)$, since any element $\gamma\in\Psi^a_a(M)/\Psi^a_{a,\zeta}(M)$ can act on $\delta\in \Psi_a(M)$ on the right by $\delta\mapsto\delta\gamma$.   

Given a homomorphism $\rho: \Gamma \to G$, we define the principal $G$-bundle $P_\rho$ as an associated $G$-bundle to $\Psi_a(M)$: 
\begin{align*}
P_\rho:=\Psi_a(M) \times_\rho G = \left(\Psi_a(M) \times G \right) \Big/(\delta\gamma,g)\sim \left(\delta,\rho(\gamma)g\right)
\end{align*}
where $\gamma \in \Gamma=\Psi^a_a(M)/\Psi^a_{a,\zeta}(M)$.  To denote the equivalence class, we use the bracket to denote a point $u=[\delta, g]\in P_\rho$.  Note that $G$ acts on $P_\rho$ by $[\delta,g]h=[\delta,gh]$ for $h\in G$. 
 
To define a connection on a principal bundle, recall that there are two ways to do so.  We can define a connection on $P_\rho$ either as a horizontal subspace, $H_u\subset (TP_\rho)_u$ at all $u\in P_\rho$, or as a one-form $A_\rho\in \Om^1(P_\rho, \mathfrak{g})$.  They are related by $H_u = \ker A_\rho\big|_u$. On $P_\rho$, we define a horizontal distribution by defining the horizontal lifts for any path $\alpha$ on $M$ as follows.
Given an arbitrary point $[\delta,g]\in (P_\rho)_{\alpha(0)}$ on the fiber of base point $\alpha(0)= \delta(1)\in M$, we define the horizontal lift of $\alpha$ starting from $[\delta,g]\in P_\rho$ by $\tilde{\alpha}(t)=[\alpha_{(t)}\delta,g]$ where $\alpha_{(t)}$ denotes the one-parameter family of paths within $\alpha$ starting at $\alpha(0)$ and ending at $\alpha(t)$.  It is straightforward to check that this construction is well-defined and satisfies the $G$ invariance condition for a connection on $P_\rho$.  We will not need to explicitly write down $A_\rho$ as a one-form in order to check that $A_\rho$ satisfies the cone-flat condition.  We will instead express the curvature in terms of the holonomy group.

\medskip

\noindent\textbf{Step 2.} Verify that the connection $A_{\rho}$ is cone-flat, that is, $F_{A_\rho} = -\zeta B_\rho$ such that $B_\rho$ is covariantly constant.

\begin{lem}
The connection $A_\rho$ constructed above is cone-flat.
\end{lem}

\begin{proof}
Let $p\in M$ be an arbitrary point and $v_1,v_2\in T_pM$ be arbitrary linearly independent vectors. Suppose $\zeta(v_1,v_2)=c\neq 0$. By Darboux's theorem, we can find a local coordinate system $\{x_1,\ldots,x_m\}$ such that $v_1=\frac{\partial}{\partial x_1}$, $v_2=\frac{\partial}{\partial x_2}$, and $\zeta=c\,dx_1\wedge dx_2+\bar{\zeta}$, where $\bar{\zeta} = dx_3 \w dx_4 + \ldots ,$ if $\dim M\geq 4$. Let $D_t$ be an infinitesimal parallelogram spanned by $\sqrt{t}v_1$ and $\sqrt{t}v_2$ in the local coordinate system, and $\gamma_t=\partial D_t$ be its boundary. Then we have $\int_{D_t} \zeta=c\,t$.

At arbitrary $[\delta,g]\in P_\rho$ on the fiber of $p$, let $v_1^H$ and $v_2^H$ be the horizontal lift of $v_1$ and $v_2$, respectively. Then the curvature
$$
F_{A_\rho}\left(v_1^H,v_2^H\right)=\frac{\partial}{\partial t}hol\left(\gamma_t\right)\Big|_{t=0}\;.
$$
Here, $hol\left(\gamma_t\right)\in G$ denotes the holonomy along $\gamma_t$ at $[\delta,g]$.

On the other hand,
$$
[\delta,g]hol\left(\gamma_t\right)=[\gamma_t\delta,g]=[\delta(\delta^{-1}\gamma_t\delta),g]=\left[\delta,g\left(g^{-1}\rho\left(\delta^{-1}\gamma_t\delta\right)g\right)\right].
$$
This implies
$$
hol\left(\gamma_t\right)=g^{-1}\rho\left(\delta^{-1}\gamma_t\delta\right)g\,.
$$
Note that $\delta^{-1}\gamma_t\delta$ is contractible and its base point is $a\in M$, i.e. $\delta^{-1}\gamma_t\delta \in \Psi^a_{a,0}(M)$. We can express $\delta^{-1}\gamma_t\delta = \exp(c\,t\, \xi)$, where $\xi$ is an element of the Lie algebra of $\Psi^a_a(M)/\Psi^a_{a,\zeta}(M)$ that 
generates the subgroup $\Psi^a_{a,0}(M)/\Psi^a_{a,\zeta}(M)\cong\mathbb{R}/H$ and  $\int_{D_t}\zeta  = \zeta(v_1, v_2)\, t= c\,t \in \mathbb{R}/H$.
Note that $\xi$ is independent of the choice of $v_1$ and $v_2$. We thus find at $p\in M$
$$
F_{A_\rho}\left(v_1^H,v_2^H\right)=\frac{\partial}{\partial t}hol\left(\gamma_t\right)\Big|_{t=0}=c\,\mathrm{Ad}_{g^{-1}}d\rho(\xi)=\zeta(v_1,v_2)\mathrm{Ad}_{g^{-1}}d\rho(\xi)
=-\zeta(v_1, v_2)B_\rho\,,$$
where $d\rho$ maps the Lie algebra of $\Gamma$ into $\mathfrak{g}=\Lie(G)$.  It is clear that $B_\rho$ as obtained above is covariantly constant as it is a constant when evaluated along horizontally lifted curves.

Thus far, we have assumed $\zeta(v_1,v_2)\neq 0$.  If however $\zeta(v_1,v_2)=0$ which may occur in $\dim M \geq 4$, then we can find a local coordinate system $\{x_1,\ldots,x_m\}$ such that $v_1=\frac{\partial}{\partial x_1}$, $v_2=\frac{\partial}{\partial x_3}$, and $\zeta=dx_1\wedge dx_2+ dx_3 \w dx_4 + \bar{\zeta}$, where $\bar{\zeta}$, if non-zero, is generated by other $dx_i\wedge dx_j$ locally. The proof goes through as above but with $c$ set to zero.  
\end{proof}

\medskip

\noindent\textbf{Step 3.} Show that the morphism $\rho\mapsto(P_{\rho},A_{\rho})$ is invariant under conjugation.

\begin{lem}
Suppose $\rho$ and $\bar{\rho}$ are conjugate homomorphisms from $\Gamma=\Psi^a_a(M)/\Psi^a_{a,\zeta}(M)$ to $G$, and also, $(P_\rho,A_\rho)$ and $(P_{\bar{\rho}},A_{\bar{\rho}})$ are principal bundles with cone-flat connections constructed as above. Then $(P_\rho,A_\rho)$ and $(P_{\bar{\rho}},A_{\bar{\rho}})$ are equivalent.
\end{lem}
\begin{proof}
Suppose $\rho$ and $\bar{\rho}$ are conjugate, i.e. $\bar{\rho}=g_0\rho g_0^{-1}$ for some $g_0\in G$.  We consider the automorphism on $\Psi_a(M) \times G$ given by $(\delta,g)\mapsto(\delta,g_0g)$.  This automorphism induces the desired bundle isomorphism 
\begin{align*}
f:\quad \Psi_a(M) \times_\rho G\quad \longrightarrow & \quad\quad \Psi_a(M) \times_{\bar\rho} G\\ 
 (\delta \gamma, g) \sim (\delta, \rho(\gamma)g) \mapsto  &\ \   (\delta\gamma,g_0g)\sim (\delta, \bar{\rho}(\gamma) g_0 g) = (\delta, g_0\rho(\gamma)g)
\end{align*}
where the second line shows the map on the equivalence class for all $\gamma \in \Gamma=\Psi^a_a(M)/\Psi^a_{a,\zeta}(M)$.  

Now to show $f^*A_{\bar\rho}=A_\rho$, let $[\delta,g]_\rho\in \Psi_a(M)\times_\rho G$ and $\alpha$ on $M$ be an arbitrary path starting at $\alpha(0)=\delta(1)$.  As described in Step 1 in defining $A_\rho$, the horizontal lift of $\alpha$ starting at $[\delta,g]_\rho$ is defined as $\tilde{\alpha}_\rho(t)=[\alpha_{(t)}\delta,g]_\rho$.  Likewise, the horizontal lift of $\alpha$ starting at $f\big( [\delta,g]_\rho \big)=[\delta,g_0g]_{\bar{\rho}}$ is defined as $\tilde{\alpha}_{\bar{\rho}}(t)=[\alpha_{(t)}\delta,g_0g]_{\bar{\rho}}$.  Clearly, we have $\tilde{\alpha}_{\bar{\rho}}(t)=f\circ \tilde{\alpha}_\rho(t)$, which implies $f_*$ sends horizontal vectors to horizontal vectors as desired.
\end{proof}

By this lemma, given any conjugacy classes $[\rho]$ of the homomorphisms from $\Psi^a_a(M)/\Psi^a_{a,\zeta}(M)$ to $G$, there is a corresponding $G$-bundle $P_{\rho}$ with a cone-flat connection $A_{\rho}$. It remains to show that this correspondence is bijective.

\medskip

\noindent\textbf{Step 4.} The morphism $[\rho]\mapsto(P_{\rho},A_{\rho})$ is injective.

\begin{lem}
Let $\rho,\bar{\rho}:\Psi^a_a(M)/\Psi^a_{a,\zeta}(M)\to G$. If $(P_\rho,A_\rho)$ and $(P_{\bar{\rho}},A_{\bar{\rho}})$ are equivalent, then $\rho$ and $\bar{\rho}$ are conjugate.
\end{lem}

\begin{proof}
Suppose $f:P_\rho \to P_{\bar\rho}$ is a $G$-bundle isomorphism and $f^*A_{\bar\rho}=A_\rho$. Let us define $h:\Psi_a(M)\to G$ such that
$$
f\big( [\delta,e]_\rho \big)=[\delta,e]_{\bar\rho}h(\delta)
$$
where $\delta\in \Psi_a(M)$ and $e$ is the identity of $G$.  Then, for each $\gamma\in \Psi^a_a(M)/\Psi^a_{a,\zeta}(M)$, we have
\begin{align*}
 [\delta,e]_{\bar\rho}h(\delta)\rho(\gamma)&=f\big( [\delta,e]_\rho \big)\rho(\gamma)=f\big( [\delta,e]_\rho\rho(\gamma) \big)\\
&= f\big( [\delta\gamma,e]_\rho \big)=[\delta\gamma,e]_{\bar\rho} h(\delta\gamma)=[\delta,e]_{\bar\rho}\bar{\rho}(\gamma)h(\delta\gamma)\,.
\end{align*}
This implies $\rho(\gamma)=h(\delta)^{-1}\bar{\rho}(\gamma)h(\delta\gamma)$. We will prove that $h$ is a constant.

For arbitrary $\delta\in \Psi_a(M)$, let $\delta_{(t)}$ denote the $t$-parametrized subpaths of $\delta$ starting at $\delta(0)=a$ and ending at $\delta(t)$.  By the construction of the connections, $[\delta_{(t)},e]_\rho$ and $[\delta_{(t)},e]_{\bar{\rho}}$ are horizontal paths on $P_\rho$ and $P_{\bar{\rho}}\,$, respectively. The projection of these two paths on $M$ are exactly $\delta_{(t)}$. By the definition of $h$, we have
\begin{align}\label{fmap1}
f\big( [\delta_{(t)},e]_\rho \big)=[\delta_{(t)},e]_{\bar\rho}h(\delta_{(t)})\,.
\end{align}
On the other hand, since $f^*A_{\bar\rho}=A_\rho$, $f\big( [\delta_{(t)},e]_\rho \big)$ is a horizontal lift of $\delta_{(t)}$ on $P_{\bar{\rho}}$ passing through $f\big( [\delta_{(0)},e]_\rho \big)$. As horizontal subspaces are invariant under right $G$-action, $[\delta_{(t)},e]_{\bar\rho}h(\delta_{(0)})$ is also a horizontal lift of $\delta_{(t)}$ on $P_{\bar{\rho}}$ passing through $[\delta_{(0)},e]_{\bar\rho}h(\delta_{(0)})=f\big( [\delta_{(0)},e]_\rho \big)$. Therefore, we have
\begin{align}\label{fmap2}
f\big( [\delta_{(t)},e]_\rho \big)=[\delta_{(t)},e]_{\bar\rho}h(\delta_{(0)})\,.
\end{align}
Comparing \eqref{fmap1}-\eqref{fmap2}, we have $h(\delta_{(t)})=h(\delta_{(0)})$ for all $t\in [0,1]$. This implies $h(\delta_{(t)})$ is equal to $h$ acting on the constant path at $a$. So $h$ is a constant, and therefore, $\rho$ and $\bar{\rho}$ are conjugate.
\end{proof}

\medskip

\noindent\textbf{Step 5.} The morphism $[\rho]\mapsto(P_{\rho},A_{\rho})$ is surjective.

The following lemma leads to surjectivity. It also holds when $\mathbb{R}/H$ is not a Lie group, which we will discuss later.
\begin{lem}\label{holonomy}
Let $\pi:P\to M$ be a principal $G$-bundle, and let $A$ be a cone-flat connection on $P$ with curvature $F_A=-\zeta B$. Then there exists $\xi\in\mathfrak{g}$ such that for any contractible loop $\gamma$ starting at $a\in M$ and any oriented disk $D$ in $M$ with $\partial D=\gamma$ ($\partial D$ and $\gamma$ also have the same orientation), the holonomy along $\gamma$ is
$$
hol(\gamma)=\mathrm{exp}\Big[\Big(\int_D \zeta\, \Big)\, \xi \Big]\,.
$$
\end{lem}
\begin{proof}
Take a point $u_0\in P$ on the fiber of $a\in M$. Consider the holonomy bundle
$$
\hat{P}=\{ u\in P\, |\, \text{there exists a horizontal curve } \tilde\delta \text{ such that } \tilde\delta(0)=u_0, \tilde\delta(1)=u \}.
$$
$\hat{P}$ is a principal bundle over $M$, and its structure group $\hat{G}$ is the holonomy group of $P$ at $u_0$. Let $\hat{A}$ be the restriction of $A$ on $\hat{P}$. By the holonomy theorem of Ambrose-Singer \cite{AS}, the Lie algebra of $\hat{G}$ is
$$
\hat{\mathfrak{g}}=\mathrm{span}\left\{ F_{\hat{A}} \!\left(v_1^H,v_2^H\right)\big| \;v_1^H,v_2^H \text{ are horizontal vectors at }u\text{ for some }u\in\hat{P} \right\}.
$$
By assumption, $F_{A} =-\zeta B$, or more precisely $F_{A} =-(\pi^*\zeta)B$. Since $B$ is covariantly constant, it is equal to some $\xi\in\hat{\mathfrak{g}}$ at any point in $\hat{P}$. Hence, $F_{\hat{A}} \left(v_1^H,v_2^H\right)\in \mathbb{R}\xi$. So $\hat{\mathfrak{g}}$ is 1-dimensional and abelian. Then $\hat{A}=-\theta \otimes\xi$ for some $\theta\in \Omega^1(\hat{P})$ and $d\theta=\hat{\pi}^*\zeta$, where  $\hat{\pi}:\hat{P}\to M$ is the projection.

For an arbitrary contractible loop $\gamma$ and a disk $D$ such that $\partial D=\gamma$, there exists a contractible neighborhood $U\subset M$ of $D$. Then $\hat{P}|_U=U\times \hat{G}$ is trivial. Let $\sigma:U\to \hat{P}|_U$ be a local section and $\psi:\hat{P}|_U\to U\times \hat{G}$ be a trivialization such that $\psi\circ\sigma(p)=(p,e)$ for $p\in M$. Then $(\psi^{-1})^*\hat{A}=(-\sigma^*\theta\otimes\xi,0)+(0,MC_{\hat{G}})$, where $MC_{\hat{G}}:T\hat{G}\to\hat{\mathfrak{g}}$ is the Maurer-Cartan form of $\hat{G}$ sending a vector to the corresponding invariant vector field. Observe that $d(\sigma^*\theta)=\zeta$. The horizontal lift $\tilde{\gamma}$ of $\gamma$ with $\tilde{\gamma}(0)=\sigma(a)$ satisfies $\psi\circ\tilde{\gamma}(t)=(\gamma(t),g(t))$ with $g(t)\in \hat{G}$ and $g(0)=e$. With the horizontal vectors in the kernel space of the connection one-form, we have
\begin{align*}
0 = \hat{A}\big(\tilde{\gamma}'(t)\big) \Big|_{\gamma(t_0)} 
&= (\psi^{-1*}\hat{A})\big(\gamma'(t),0\big) \Big|_{\big(\gamma(t_0),g(t_0)\big)}+(\psi^{-1*}\hat{A})\big(0,g'(t)\big) \Big|_{\big(\gamma(t_0),g(t_0)\big)} \\
&= -(\sigma^*\theta)\big(\gamma'(t)\big)\cdot \xi \Big|_{\gamma(t_0)}+MC_{\hat{G}}\big(g'(t)\big) \Big|_{g(t_0)}.
\end{align*}
Hence, we find
$$
g(t_0)=\mathrm{exp}\Big[ \Big(\int_0^{t_0}(\sigma^*\theta)\big(\gamma'(t)\big)dt\,\Big)\,\xi \Big],
$$
and the holonomy along $\gamma$ is
$$
hol(\gamma)=g(1)=\mathrm{exp}\Big[ \Big(\int_{\gamma}\sigma^*\theta \,\Big)\, \xi \Big]=\mathrm{exp}\Big[\Big(\int_D \zeta\, \Big)\, \xi \Big].
$$
\end{proof}

Now we proceed to prove surjectivity.

\begin{lem}
Let $\pi:P\to M$ be a principal $G$-bundle with a cone-flat connection $A$ with respect to $\zeta$.  Then there exists a homomorphism $\rho:\Psi^a_a(M)/\Psi^a_{a,\zeta}(M)\to G$ such that $(P_\rho,A_\rho)$ and $(P,A)$ are equivalent.
\end{lem}
\begin{proof}
For any contractible loop $\gamma\in\Psi^a_{a,\zeta}(M)$ and a disk $D$ such that $\partial D=\gamma$, it follows from the definition of $\Psi^a_{a,\zeta}(M)$ that $\int_D \zeta=0$. By Lemma $\ref{holonomy}$, the holonomy along $\gamma$ in $(P, A)$ is given by $hol(\gamma)=\mathrm{exp}\Big[\Big(\int_D \zeta\, \Big)\, \xi \Big]=e$. Hence, we define the following homomorphism:
\begin{align}\label{hol}
\rho:\Psi^a_a(M)/\Psi^a_{a,\zeta}(M)\to G,\quad \gamma\mapsto \rho(\gamma)=hol(\gamma)\,.
\end{align}
We will show that the resulting $(P_\rho,A_\rho)$ is equivalent to $(P,A)$.  We start by defining a $G$-bundle isomorphism between the two principal bundles.

Given $[\delta,g]\in P_\rho$, let $\tilde{\delta}$ be the horizontal lift of $\delta$ in $P$ with $\tilde{\delta}(0)=u_0$. We define
\begin{align}\label{fdeff}
f:P_\rho\to P,\quad [\delta,g]\mapsto \tilde{\delta}(1)\,g\,.
\end{align}
Note that the definition of $f$ utilizes the horizontal lift determined by the connection $A$ in $P$. 
Let us show that this map is well-defined.  Suppose $[\delta_1,g_1]$ and $[\delta_2,g_2]$ represent the same class in $P_\rho$.  Then there must exists a $\gamma_0\in \Psi^a_a(M)/\Psi^a_{a,\zeta}(M)$ such that 
\begin{align} \label{fwdef}
(\delta_1, g_1)= (\delta_2\gamma_0, \rho(\gamma_0^{-1})g_2) \sim (\delta_2, g_2)\,.
\end{align}
Thus, $\gamma_0^{-1}\delta_2^{-1}\delta_1\in \Psi^a_{a,\zeta}(M)$, and its holonomy  in $(P, A)$ is trivial according to Lemma \ref{holonomy}. Moreover, we have 
\begin{align}\label{hhrho}
 hol(\delta_2^{-1}\delta_1)= hol(\gamma_0)=\rho(\gamma_0)\,,
\end{align}
where the last equality follows from our definition of $\rho$ in \eqref{hol}.
Now let $\tilde{\delta}_1$ and $\tilde{\delta}_2$ denote the horizontal lift in $(P, A)$ of $\delta_1$ and $\delta_2$, respectively, starting at $u_0$. Since $\tilde{\delta}_1$ and the horizontal lift of $\delta_2\gamma_0=\delta_2(\delta_2^{-1}\delta_1)$ starting at $u_0$ have the same ending point, so does $\tilde{\delta}_1$ and the horizontal lift of $\delta_2$ starting at $u_0\cdot hol(\gamma_0)$. Hence, the ending points of $\tilde{\delta}_1$ and $\tilde{\delta}_2$ satisfy $\tilde{\delta}_1(1)=\tilde{\delta}_2(1)\cdot hol(\gamma_0)$. Together with \eqref{fdeff}-\eqref{hhrho}, we find
$$
f\left([\delta_1,g_1]\right)=\tilde{\delta}_1(1)\,g_1=\tilde{\delta}_2(1)\rho(\gamma_0)\,g_1=\tilde{\delta}_2(1)\,g_2=f(\left[\delta_2,g_2]\right).
$$
Hence, $f$ is well-defined.  It is also straightforward to check that $f$ is a $G$-bundle isomorphism.

Finally, we check that the definition of $A_\rho$ described earlier in Step 1 is consistent with $f^*A=A_\rho$. For $[\delta,g]\in P_\rho$ and an arbitrary path $\alpha$ on $M$  such that $\alpha(0)=\delta(1)$, the horizontal lift of $\alpha$ at $[\delta,g]$ is $\tilde{\alpha}(t)=[\alpha_{(t)}\delta,g]$. Here again, $\alpha_{(t)}$ is the subpath of $\alpha$ starting at $\alpha(0)$ and ending at $\alpha(t)$. Notice that $f\left([\alpha_{(t)}\delta,g]\right)= (\widetilde{\alpha_{(t)}\delta})(1)\,g$ is the ending point of the path on $P$ that is the horizontal lift of $\alpha_{(t)}$ in $(P, A)$ starting at $f\left([\delta,g]\right)=\tilde{\delta}(1)\,g$. Clearly then, $f\circ\tilde{\alpha}$ is a horizontal path on $(P, A)$.  Hence,  $f_*$ sends horizontal vectors to horizontal vectors, and therefore, $f^*(A)=A_\rho$.
\end{proof}

Combining the lemmas above, we have proved that $[\rho]\mapsto(P_{\rho},A_{\rho})$ is an isomorphism.

\bigskip

\noindent \textbf{Proof for Case 3: $\mathbb{R}/H$ is not a Lie group}

We now turn to the case when $\mathbb{R}/H$ is not a Lie group. In this case, $H^+$ is non-empty and has no minimal number.
\begin{lem} 
When $H^+$ is non-empty and has no minimal number, a cone-flat connection is a flat connection.
\end{lem}
\begin{proof}
By Lemma \ref{holonomy}, there exists some $\xi\in\mathfrak{g}$ such that for any contractible loop $\gamma$ and disk $D$ with $\partial D=\gamma$, we have $hol(\gamma)=\mathrm{exp}\Big[\Big(\int_D \zeta\, \Big)\, \xi \Big]$.  

Let us show that $\xi$ must be zero.  For if $\xi\neq 0$, then there exists some small enough $t_0$ such that for any $0<t<t_0\,$, $\mathrm{exp}(t\,\xi)\neq e\,$ the identity element of $G$.  Now since $H^+$ has no minimal number, there exists a closed sphere $S\subset M$ such that $\int_S\zeta=t$ for some $0<t<t_0$.  Let $\gamma$ be a constant loop at some point $p\in S$, and $D=S\setminus\!\{p\}\,$.  Then $\partial D=\gamma$ so that $hol(\gamma)=\mathrm{exp}\Big[\Big(\int_D \zeta\, \Big)\, \xi \Big]$. But $hol(\gamma)=e$ as $\gamma$ is the identity loop. On the other hand, $\int_D \zeta=\int_S\zeta=t\,$, and therefore, $\mathrm{exp}(t\,\xi)=e$, which gives a contradiction. Thus, we conclude that $\xi$ is zero.  

With $\xi=0$, the holonomy of any contractible loop is trivial.  Hence, the curvature vanishes and the connection is flat.
\end{proof}

The classification of $G$-bundles with flat connections can be represented by the conjugacy classes of homomorphisms from $\pi_1(M)\to G$ (c.f. \cite{Morita}*{Theorem 2.9}). So we have proved Theorem \ref{classification} in this case. This completes the proof of the theorem.

\

We point out that the classification of cone-flat bundles generally depends on the choice of $\zeta$.  In Theorem \ref{classification}, the $\zeta$ dependence explicitly appears in the definition of $H$ in \eqref{Hdef}.  Below, we will work out the classification and demonstrate its dependence on $\zeta$ in the simple example of $U(1)$ bundles over the four-dimensional torus $M=T^4$.

\begin{ex}
We describe $T^4$ as $\mathbb{R}^4/\sim$, with the identification $(x_1,x_2,x_3,x_4)\sim(x_1+a,x_2+b,x_3+c,x_4+d)$ where $a,b,c,d\in\mathbb{Z}$, and $U(1)$ as $\{ z\in\mathbb{C}\mid |z|=1 \}$. Let 
$$\zeta=c_1\,dx_1\wedge dx_2+c_2\, dx_3\wedge dx_4$$ 
be a closed 2-form with $c_1,c_2\in\mathbb{R}\setminus \{0\}$. By Theorem \ref{classification}, the equivalent classes of $U(1)$ bundles with a cone-flat connection are in 1-1 correspondence with homomorphisms $\rho:\Gamma\to U(1)$, where $\Gamma=\Psi^a_a(M)/\Psi^a_{a,\zeta}(M)$. 

Since $\pi_2(T^4)$ is trivial, $\Gamma$ is an $\mathbb{R}$-extension of $\pi_1(T^4)$. To describe its group structure explicitly, let $a_i$ for $i=1,2,3,4$ be the straight line path in $\mathbb{R}^4$ starting at the origin and ending at the point where the $i$-th coordinate, $x_i=1\,$, and $x_j=0$ for $j\neq i\,$. When projected to $T^4$, $\{a_1, a_2, a_3, a_4\}$ become the generators of $\pi_1(T^4)$. Although $\pi_1(T^4)$ is abelian, the elements of $\{a_1, a_2, a_3, a_4\}$ may no longer commute in $\Gamma$. 

For a contractible loop $b$, there is an  oriented disk $D$ such that $\partial D=b$ and $\partial D$ has the same orientation as $b$.  Let $|b|=\int_D \zeta$, and we note that $|b|$ is independent of the choice of $D$.  Now recall that  $ \Gamma=\Psi^a_a(M)/\Psi^a_{a,\zeta}(M)\,$ is a quotient of loops by contractible ones whose integral, $|b|=\int_D \zeta\,=0\,$.   Moreover, contractible loops $b$ and $b'$ would represent different classes in $\Gamma$ if and only if $|b|\neq |b'|$.  So, a class of contractible loops $b\in\Gamma$ can be represented by the real number $|b|$.  

On $T^4$ then, we can think of $\Gamma$ as being generated by $\{a_1,a_2, a_3,a_4,|b|\}$, with multiplication defined by
\begin{align*}
a_i|b|=|b|a_i\, \qquad |b'b|=|b|+|b'|\,.    
\end{align*}
However, $a_j^{-1}a_i^{-1}a_ja_i$ can represent some non-trivial class of a contractible loop.  Specifically, we have
\begin{align*}
|a_2^{-1}a_1^{-1}a_2a_1|&=c_1, \\
|a_4^{-1}a_3^{-1}a_4a_3|&=c_2, \\
|a_j^{-1}a_i^{-1}a_ja_i|&=0, \text{ for other }i,j,  
\end{align*}

Now, the possible homomorphisms of $\rho:\Gamma\to U(1)$ is dependent on whether $\frac{c_1}{c_2}\in\mathbb{Q}$. Since $U(1)$ is abelian, 
$$
\rho(c_1)=\rho(a_2^{-1})\rho(a_1^{-1})\rho(a_2)\rho(a_1)=\rho(a_2)^{-1}\rho(a_2)\rho(a_1)^{-1}\rho(a_1)=1.
$$
Similarly $\rho(c_2)=1$, so $\rho(pc_1+qc_2)=1$ for any $p,q\in\mathbb{Z}$.

When $\frac{c_1}{c_2}\in\mathbb{Q}$, $\{ (pc_1+qc_2)\,|\, p,q\in\mathbb{Z} \}$ has a minimal positive number $c_0$. Then $\rho(b)$ must have the form $e^{2\pi i\frac{ n|b|}{c_0}}$ for some $n\in\mathbb{Z}$. In this case, the Euler class of the circle bundle is $\frac{n}{c_0}\zeta$, and the choice of $\rho(a_i)$ for $i=1, \dots, 4$ determines the connection.

When $\frac{c_1}{c_2}\notin\mathbb{Q}$, $\{ (pc_1+qc_2)\,|\, p,q\in\mathbb{Z} \}$ is dense in $\mathbb{R}$. So $\rho(b)$ must be 1 for any $|b|\in\mathbb{R}$. Thus, the classification is only dependent on $\rho(a_i)$, and becomes equivalent to the classification of flat connections. Actually, the Euler class in this case is $c\,\zeta$ for some $c\in\mathbb{R}$.  But the Euler class is an integral class, and hence, the only possible $c$ is 0, i.e. every cone-flat connection is flat.
\end{ex}

\medskip

\begin{rmk}
When $M$ is simply-connected,
the Hurewicz homomorphism $\pi_2(M)\to H_2(M,\mathbb{Z})$ is surjective. Assuming $\zeta$ is not $d$-exact, $H$ is discrete if and only if there exists some non-zero constant $c$ such that $c[\zeta]\in H^2(M,\mathbb{Z})$. Therefore, if such $c$ exists, then $\Gamma=\mathbb{R}/\mathbb{Z}$ and the classification of isomorphism classes of cone-flat connections on $G$-bundles is given by the conjugacy classes of $\text{Hom}(S^1,G)$. If such $c$ does not exist, then $\Gamma$ is trivial and then there does not exist non-trivial cone-flat connections.

 As an application of this observation, let $M$ be simply-connected and closed.  If $M$ is also a projective manifold, then we can choose $\zeta$ to be the K\"ahler form which is an  integral class, and so, $\zeta\in H^2(M, \mathbb{Z})$.  The isomorphism classes of cone-flat connections on $G$-bundles is then given by the conjugacy classes of $\text{Hom}(S^1,G)$.  If however $M$ is a non-projective K\"ahler manifold and we still let $\zeta$ be the K\"ahler form, then there is no non-zero constant $c$ such that $c[\zeta]\in H^2(M,\mathbb{Z})$ (see, for a reference, \cite{Huybrechts}*{Corollary 5.3.3}) and the isomorphism classes of cone-flat connections would be trivial.
\end{rmk}

\

\noindent {\bf Data Availability} 
There is no data associated with this article.

\

\noindent{\bf Conflict of Interest} 
The authors have no financial or proprietary interests in any material discussed in this article.

\



\vskip 1cm
\noindent
{Department of Mathematics, University of California, Irvine, CA 92697, USA}\\
{\it Email address:}~{\tt lstseng@uci.edu}
\vskip .5 cm
\noindent
{School of Mathematics and Computer Sciences, Nanchang University, Nanchang, Jiangxi, 330031, China}\\
{Beijing Institute of Mathematical Sciences and Applications, Beijing, 101408, China}\\
{\it Email address:}~{\tt jiaweizhou90@ncu.edu.cn}


\begin{bibdiv}
\begin{biblist}[\normalsize]

\bib{AS}{article}{
author={Ambrose, W.},
   author={Singer, I. M.},
   title={A theorem on holonomy},
   journal={Trans. Amer. Math. Soc.},
   volume={75},
   date={1953},
   pages={428--443},
   issn={0002-9947},
}

\bib{Atiyah-Bott}{article}{
   author={Atiyah, M. F.},
   author={Bott, R.},
   title={The Yang-Mills equations over Riemann surfaces},
   journal={Philos. Trans. Roy. Soc. London Ser. A},
   volume={308},
   date={1983},
   number={1505},
   pages={523--615},
   issn={0080-4614},
}

\bib{CTT}{article}{
    author={Clausen, D.},
    author={Tseng, L.-S.},
    author={Tang, X.},
    title={Mapping cone and Morse theory},
    note = {arXiv:2405.02272 [math.DG]},
}

\bib{Huybrechts}{book}{
   author={Huybrechts, D.},
   title={Complex Geometry: An Introduction},
   publisher={Springer-Verlag, Berlin},
   date={2005},
   pages={xii+309},
}




\bib{LeBrun}{article}{
   author={LeBrun, C.},
   title={Complete Ricci-flat K\"{a}hler metrics on ${\bf C}^n$ need not be
   flat},
   conference={
      title={Several complex variables and complex geometry, Part 2},
      address={Santa Cruz, CA},
      date={1989},
   },
   book={
      series={Proc. Sympos. Pure Math.},
      volume={52, Part 2},
      publisher={Amer. Math. Soc., Providence, RI},
   },
   date={1991},
   pages={297--304},
}

\bib{Morita}{book}{
   author={Morita, S.},
   title={Geometry of Characteristic Classes},
   series={Translations of Mathematical Monographs},
   volume={199},
   publisher={American Mathematical Society, Providence, RI},
   date={2001},
   pages={xiv+185},
   isbn={0-8218-2139-3},
}



\bib{Morrison}{article}{
   author={Morrison, K.},
   title={Yang-Mills connections on surfaces and representations of the path
   group},
   journal={Proc. Amer. Math. Soc.},
   volume={112},
   date={1991},
   number={4},
   pages={1101--1106},
   issn={0002-9939},
}

\bib{Oh}{article}{
   author={Oh, J. J.},
   author={Park, C.},
   author={Yang, H. S.},
   title={Yang-Mills instantons from gravitational instantons},
   journal={J. High Energy Phys.},
   date={2011},
   number={4},
   pages={087, 37},
}
\bib{Prasad}{article}{
   author={Prasad, M. K.},
   title={Instantons and monopoles in Yang-Mills gauge field theories},
   journal={Phys. D},
   volume={1},
   date={1980},
   number={2},
   pages={167--191},
   issn={0167-2789},
}





\bib{TZ}{article}{
   author={Tseng, L.-S.},
   author={Zhou, J.},
   title={Symplectic flatness and twisted primitive cohomology},
   journal={J. Geom. Anal.},
   volume={32},
   date={2022},
   number={282},
}

\bib{Weibel}{book}{
   author={Weibel, J.},
   title={An Introduction to Homological Algebra},
   series={Cambridge Series in Advanced Mathematics},
   publisher={Cambridge University Press, New York, N.Y.},
   year={1994},
   pages={1--25},
   doi ={10.1017/CBO9781139644136},
}


\end{biblist}
\end{bibdiv}
\end{document}